\documentclass[12pt]{article}
\usepackage{graphicx,amsmath,amssymb}
\usepackage{color}
\makeatletter
 \renewcommand{\theequation}{%
    \thesection.\arabic{equation}}
 \@addtoreset{equation}{section}
\makeatother

\newcommand{\qed}{\hfill\framebox(6,8){}}
\newtheorem{thm}{Theorem}[section]
\newtheorem{lem}[thm]{Lemma}
\newtheorem{defn}{Definition}[section]
\newtheorem{cor}[thm]{Corollary}
\newtheorem{prop}[thm]{Proposition}

\newcommand{\bR}{{\mathbb R}}

\def\eproof{\hfill{\vrule height5pt width3pt depth0pt}}

\begin{document}
\title{{\bf Multiple front standing waves  
in the FitzHugh-Nagumo equations}}

\author{Chao-Nien Chen \thanks{Department of Mathematics, National Tsing Hua University, Hsinchu 300, Taiwan ({\tt chen@math.nthu.edu.tw)}}
\and {\'Eric S\'er\'e}\thanks{CEREMADE, Universit\'e Paris-Dauphine, PSL Research University, CNRS, UMR 7534, Place de Lattre de Tassigny, F-75016 Paris, France 
({\tt sere@ceremade.dauphine.fr)}}}

\date{}
\maketitle

\begin{quote}
\textbf{Abstract}:
There have been several existence results for the standing
waves of FitzHugh-Nagumo equations. Such waves are the connecting orbits
of an autonomous second-order Lagrangian system and the corresponding kinetic
energy is an indefinite quadratic form in the velocity terms. When the
system has two stable hyperbolic equilibria, there exist two stable standing
fronts, which will be used in this paper as building blocks, to construct
stable standing waves with multiple fronts in case the equilibria are of
saddle-focus type. The idea to prove existence is somewhat close in spirit
to \cite{BS}; however several differences are required in the argument:
facing a strongly indefinite functional, we need to perform a nonlocal
Lyapunov-Schmidt reduction; in order to justify the stability of multiple
front standing waves, we rely on a more precise variational characterization
of such critical points. Based on this approach, both stable and unstable
standing waves are found.


\end{quote}

\newpage

\section{Introduction} \label{sec_intro}
\setcounter{equation}{0}
\renewcommand{\theequation}{\thesection.\arabic{equation}}
Following a fascinating idea of Turing \cite{T}, reaction-diffusion systems
\cite{BL,CCHg,CC1,CET,F,L,NAY} serve as models for studying pattern formation
and wave propagation. Significant progress
\cite{CC,CCF,CCHR,CCR,CH2,CKM,CT,CJM,DY2,DHK,O,RS,RW,vCNT,WW} on the self-organized patterns
has been made for the system of FitzHugh-Nagumo equations
%
\begin{eqnarray}
\label{eq_FN_u}
&&u_{t}-du_{xx}=f(u)-v,
\\
\label{eq_FN_v}
&&\tau v_{t}- v_{xx}=u-\gamma v.
\end{eqnarray}
Here $f(\xi )=\xi (\xi -\beta )(1-\xi )$,
${\displaystyle \beta \in (0,{1}/{2})}$ and
$d, \tau , \gamma \in (0,\infty )$. Historically the original model
\cite{F,NAY} was derived as a simplification of the Hodgkin-Huxley equations
\cite{HH} for nerve impulse propagation. In recent years \eqref{eq_FN_u}-\eqref{eq_FN_v} has been extensively studied as a paradigmatic
activator-inhibitor system. Such systems are of great interest to the scientific
community as breeding grounds for studying the generation of localized
structures.
\smallskip

The standing wave solutions of \eqref{eq_FN_u}-\eqref{eq_FN_v} are the
connecting orbits of a second order Lagrangian system
%
\begin{eqnarray}
\label{eq_u}
&&-d u''=f(u)-v,
\\
\label{eq_v}
&&-v'' =u-\gamma v.
\end{eqnarray}
Associated with \eqref{eq_u}-\eqref{eq_v}, the Lagrangian is
%
\begin{eqnarray}
\label{La}
{L(u,v,u',v')}&=&\frac{d}{2}(u')^{2}-\frac{1}{2}(v')^{2}+uv-
\frac{\gamma }{2}v^{2}-\int _{0}^{u} f(\xi )d\xi .
\end{eqnarray}
Since \eqref{eq_u}-\eqref{eq_v} is an autonomous Lagrangian system, the
associated energy
%
\begin{eqnarray}
\label{Eg}
E(u,v,u',v'):=\frac{d}{2}(u')^{2}-\frac{1}{2}(v')^{2}-uv+
\frac{\gamma }{2}v^{2}+\int _{0}^{u} f(\xi )d\xi
\end{eqnarray}
is constant along any solution. Expressed in terms of the momenta
$p=du'$ and $q=-v'$, this energy becomes the Hamiltonian
%
\begin{eqnarray}
\label{Hm}
H(u,v,p,q):=\frac{1}{2d}p^{2}-\frac{1}{2}q^{2}-uv+\frac{\gamma }{2}v^{2}+
\int _{0}^{u} f(\xi )d\xi \;.
\end{eqnarray}
The Hamiltonian system associated with $H$ in the phase space
${\mathbb{R}}^{2}\times{\mathbb{R}}^{2}$ will be referred to as (HS). It is a reformulation of \eqref{eq_u}-\eqref{eq_v} as a system of first-order equations.
\smallskip

When $\gamma$ is not too small, system \eqref{eq_u}-\eqref{eq_v} has two hyperbolic equilibria
$(u_{-},v_{-})=(0,0)$ and $(u_{+},v_{+})$ with $u_{+}>0$, that are stable for \eqref{eq_FN_u}-\eqref{eq_FN_v}. Since we are
interested in standing front solutions joining such two equilibria, they
must have the same energy $E$, and this imposes the condition
$\gamma =9(2\beta ^{2}-5\beta +2)^{-1}$, as to be a basic assumption of
the paper. Under this assumption, it is easy to see that
$(u_{+},v_{+})=(2(\beta +1)/3,2(\beta +1)/3\gamma )$ and the symmetry
$S$ with respect to the center $(u_{+}/2,v_{+}/2)$ preserves the Lagrangian.
As a consequence, the linearized equations at $(u_{-},v_{-})$ and $(u_{+},v_{+})$ are the same.
\smallskip

Recall from \cite{CKM} that $\hat{v}$ is a $C^{\infty }$-function satisfying
%
\begin{eqnarray}
\hat{v}(x)=\left \{
\begin{array}{c@{\quad }c@{\quad }c}
v_{+} & \text{ for } & x \geq 1
\\
0& \text{ for } & x\leq -1.%
\end{array}
\right .
\end{eqnarray}

To show the existence of standing front solutions of \eqref{eq_u}-\eqref{eq_v},
we work with affine functional spaces of the form
$H_{w}=w+H^{1}({\mathbb{R}})$, with $w=\hat{v}, \hat{u}$, where
$\hat{u}=-\hat{v}''+\gamma \hat{v}$.

A crucial point is that \eqref{eq_v} is a linear equation allowing to solve
$v$ from $u$. First, given any $\phi \in H^{1}({\mathbb{R}})$, we denote
${\mathcal{L}}\phi $ the unique solution, in $H^{1}({\mathbb{R}})$, of the equation
%
\begin{eqnarray}%
\label{Lphi}
-g''+\gamma g=\phi .
\end{eqnarray}
Direct calculation shows that ${\mathcal{L}}$ is a bounded self-adjoint operator
in $L^{2}({\mathbb{R}})$. Then, for $u \in H_{\hat{u}}$, we denote
${\mathcal{L}} u:=\hat{v}+{\mathcal{L}} \phi $ where
$\phi =u-\hat{u}\in H^{1}({\mathbb{R}})$. By construction, ${\mathcal{L}} u$ is
the unique solution, in $H_{\hat{v}}$, of the equation
\begin{equation*}
-v''+\gamma v= u.
\end{equation*}
For $u \in H_{\hat{u}}$, define
%
\begin{eqnarray}%
\label{J(u)0}
J(u)=\int _{-\infty }^{\infty }L(u,{\mathcal{L}}u,u',{\mathcal{L}}u')dx\,.
\end{eqnarray}
It is not difficult to verify that
%
\begin{eqnarray}%
\label{minmax0}
J(u)&=& \max _{v \in H_{\hat{v}}}\int _{-\infty }^{\infty }L(u,v,u',v') dx.
\end{eqnarray}
In \cite{CKM} the action functional $J$ has been employed, through a minimization
argument, to obtain a basic type standing front solution of \eqref{eq_FN_u}-\eqref{eq_FN_v} as follows.
%
\begin{thm}
\label{thm:exist}
If ${\displaystyle \beta \in (0,{1}/{2})}$, $\gamma =9(2\beta ^{2}-5\beta +2)^{-1}$ and $\,d>\gamma ^{-2}$, there
exists a standing front solution $(u^{*},v^{*})$ of \eqref{eq_FN_u}-\eqref{eq_FN_v}
with asymptotic behavior $(u^{*},v^{*})\to (u_{-},v_{-}) $ as
$x\to -\infty $ and
${\displaystyle (u^{*},v^{*})\to (u_{+},v_{+}) }$ as $x\to \infty $. Moreover,
$u^{*}$ is a minimizer of $J$ over $H_{\hat{u}}$.
\end{thm}

If we define $(u_{*}(x)$, $v_{*}(x)):=(u^{*}(-x),v^{*}(-x))$, then $(u_{*},v_{*})$ is also a standing front solution of \eqref{eq_FN_u}-\eqref{eq_FN_v}, by reversibility of \eqref{eq_u}-\eqref{eq_v}.
\smallskip

The goal of this paper is to construct multiple front solutions using
$(u^{*},v^{*})$ together with the reverse orbit $(u_{*},v_{*} ) $. We only
deal with the case when the equilibria are of saddle-focus type; that is,
the linearization of the Hamiltonian system (HS) at $(u_{-},v_{-},0,0)$, as well as
$(u_{+},v_{+},0,0)$, has four complex eigenvalues $\pm \lambda \pm i\omega $ with $\lambda>0$ and $\omega>0$. As to be seen
in the Appendix, this additional condition holds when $\vert \gamma^2d-2-\beta\gamma \vert < 2\sqrt{1+\beta\gamma}$.
Summarizing, our assumptions on $\beta,\,\gamma,\,d$ are the following:
%
\begin{equation}%
\label{Hy}
\begin{split}
&0<\beta <1/2\,,\ \ \gamma =9/(2\beta ^{2}-5\beta +2)\,,\\
&\frac{1}{\gamma^2 }\max(1,2+\beta\gamma-2\sqrt{1+\beta\gamma}) < d < \frac{2+\beta\gamma+2\sqrt{1+\beta\gamma}}{\gamma^2 }\,.
\end{split}
\end{equation}

We now state the main existence result of the paper.

\begin{thm}
\label{thm:multi}
Assume that (\ref{Hy}) is satisfied. Then there are two real numbers
$\kappa _{+},\,\kappa _{-}$, and, for each small
$\sigma >0$, a large constant $\Delta_{\sigma }>0$, such that for any positive
integer $N$ and any sequence of positive integers
${\mathbf{{n}}}=(n_{i})_{1\leq i\leq N}$ with $n_{i}\geq \Delta_{\sigma }$ for every
$i$, there exist positive numbers $X_{1},\cdots ,{X_{N}}$ and a solution
${(\hat{u}_{\mathbf{{n}}},\hat{v}_{\mathbf{{n}}})}$ of \eqref{eq_u}-\eqref{eq_v}
satisfying the following properties:
\smallskip

\noindent
{(a)} For $i$ odd in $[1,N]$,
\begin{equation*}
\Vert (\hat{u}_{\mathbf{{n}}},\hat{v}_{\mathbf{{n}}})(\cdot +C_{i}) - (u_{*},v_{*})
\Vert _{H^{1}(-A_{i},A_{i+1})} < \sigma \,,\; \vert X_{i}-2\pi n_{i}
/\omega -\kappa _{+}\vert <\sigma \,.
\end{equation*}
(b) For $i$ even in $[0,N]$,
\begin{equation*}
\Vert (\hat{u}_{\mathbf{{n}}},\hat{v}_{\mathbf{{n}}})(\cdot +C_{i}) - (u^{*},v^{*})
\Vert _{H^{1}(-A_{i},A_{i+1})} < \sigma \,,\; \vert X_{i}-2\pi n_{i}
/\omega -\kappa _{-}\vert <\sigma \,.
\end{equation*}
Here, $A_{0}=+\infty $, $C_{0}=0$, $C_{i}=C_{i-1}+X_{i}$,
$A_{i}=X_{i}/2$ for $1\leq i\leq N$, and $A_{N+1}=+\infty $.
\end{thm}

Let us remark that if $N$ is odd,
$(\hat{u}_{\mathbf{{n}}},\hat{v}_{\mathbf{{n}}})$ is homoclinic to
$(u_{-},v_{-})$ while for $N$ even, it is a heteroclinic connection between
$(u_{-},v_{-})$ and $(u_{+},v_{+})$. Such orbits are the standing waves
of \eqref{eq_FN_u}-\eqref{eq_FN_v} with multiple fronts; for the Hamiltonian
system they are often called multibump solutions.
\smallskip

As already mentioned, the range of parameters under consideration is such
that the basic heteroclinics $(u_{*},v_{*})$ and $(u^{*},v^{*})$ connect
two equilibria of saddle-focus type. In this situation, multi-bump solutions
are known to exist provided the stable and unstable manifolds intersect
transversally, as was proved by Devaney \cite{D} by constructing a Smale
horseshoe. Transversality condition in general is difficult to check for
a given Hamiltonian although it is generically true. Instead of verifying
transversality, we follow a strategy introduced in \cite{BS}. We first
prove that any critical point of $J$ is isolated up to translation invariance
in the spatial variable, by solving an auxiliary boundary value problem.
Then we invoke this property to show the existence of multi-bump solutions
by a variational argument.%

The variational construction for multibump and chaotic solutions has a
long history and the comments below are not exhaustive. To our knowledge,
the earliest results were established by Bolotin
\cite{Bol1,Bol2,Bol3} in the context of nonautonomous second order Lagrangian
systems, the connecting orbits being minimizers of the action. In the case
of twist maps on the annulus (also corresponding to non-autonomous Lagrangian
systems), Mather \cite{Ma} constructed chaotic connecting orbits by a minimization
method in the region between two invariant circles. For non-autonomous
first order Hamiltonian systems, multibump solutions were found by min-max
methods \cite{S1,S2} under the assumption that critical points are isolated.
This construction was extended to second order systems and elliptic PDEs
in \cite{CZR1,CZR2,CTz}. We refer to \cite{R1} and references therein for more
recent development and related results in this direction. For autonomous
problems of saddle-focus type a class of multi-bump solutions were obtained,
in the special case of a fourth order equation related to water wave
theory, by Buffoni \cite{Bu} using a shooting argument. Subsequently a
larger set of multi-bump solutions was constructed \cite{BS} by variational
and degree arguments. This method was then adapted for studying the extended
Fisher-Kolmogorov equations (of fourth order) \cite{KV}. In subsequent
works \cite{KKV1,KKV2}, a refined but more specific argument was introduced
in order to obtain optimal results on the F-K model. As already mentioned,
the present work is close in spirit to \cite{BS}. Note, however, that our
system of autonomous second order Lagrangian equations is associated with
a strongly indefinite variational problem and to our knowledge, it cannot
be reduced to a fourth order equation which would allow a simpler variational
interpretation. Instead, we use a nonlocal Lyapunov-Schmidt reduction.
Moreover our approach is purely variational, contrary to \cite{BS} where
degree theory was employed. Another novelty is our proof that \textit{all}
critical points are isolated up to translations in $x$, while in
\cite{BS} the first step just consisted in showing that the basic one-bump
solution is isolated. The purely variational construction and the stronger
isolatedness property are needed for the sake of stability analysis, as
always an important issue in considering pattern formation as well as wave
propagation.
\smallskip

For the stationary solutions of \eqref{eq_FN_u}-\eqref{eq_FN_v}, stability
questions have been studied in \cite{CH1,CH2,CKM,CJM,O,Y3} by various methods.
In conjunction with the strongly indefinite variational structure, the
Maslov index \cite{CH,CH2} and relative Morse index \cite{CH1} provide
useful information to determine the stability of such solutions, obtained
as the critical points of the action functional. Let
${\mathbb{C}}^{-}=\{\zeta | \zeta \in {\mathbb{C}} \text{ and } Re \zeta <0 \}$,
where $Re \zeta $ denotes the real part of $\zeta $. Denote by
$\Lambda $ the linearization of \eqref{eq_FN_u}-\eqref{eq_FN_v} at a standing
wave solution $(u,v)$. A standing wave $(u,v)$ is said to be non-degenerate
if zero is a simple eigenvalue of $\Lambda $.\newpage

\begin{defn}
A non-degenerate standing wave $(u,v)$ of \eqref{eq_FN_u}-\eqref{eq_FN_v} is spectrally stable if all the non-zero
eigenvalues of $\Lambda $ are in ${\mathbb{C}}^{-}$.
\end{defn}

The following result follows immediately from an index method developed
in \cite{CCH}:

\begin{thm}
\label{LS}
Let $(u,v)$ be a non-degenerate standing wave of \eqref{eq_FN_u}-\eqref{eq_FN_v}.
Suppose $u$ is a local minimizer of $J$ then $(u,v)$ is spectrally stable,
provided that $\tau <\gamma ^{2}$.
\end{thm}

Note, however, that the non-degeneracy of a standing wave is equivalent
to the transversality of the stable and unstable manifolds, and we are
unable to prove such a property. Fortunately, we can go beyond the spectral
stability analysis, thanks to a Lyapunov functional introduced in
\cite{CJM} in a slightly different context. In Section~\ref{ly} we shall
give an extension of this Lyapunov functional, which can be applied to
the standing waves of \eqref{eq_FN_u}-\eqref{eq_FN_v}. Let us remark that
the standing waves are in affine subspaces of
$H_{loc}^{1}(\mathbb{R})\times H_{loc}^{1}(\mathbb{R})$ having
$H^{1}(\mathbb{R})\times H^{1}(\mathbb{R})$ as underlying vector space. The norm
of $H^{1}(\mathbb{R})\times H^{1}(\mathbb{R})$ induces the natural metric on such
affine spaces, and we shall study the dynamical stability of the standing
waves for this metric.
%
\begin{thm}
\label{thm:stab}
Assume that (\ref{Hy}) is satisfied and let $\tau <\gamma ^{2}$. Under the flow generated by \eqref{eq_FN_u}-\eqref{eq_FN_v}
on the affine space
$(\hat{u}_{\mathbf{{n}}}+H^{1}({\mathbb{R}})\,)\times (\hat{v}_{\mathbf{{n}}}+H^{1}({
\mathbb{R}})\,)$, the standing wave
$(\hat{u}_{\mathbf{{n}}},\hat{v}_{\mathbf{{n}}})$ is asymptotically stable for the
$H^{1}({{\mathbb{R}}})\times H^{1}({{\mathbb{R}}})$ metric, up to a phase
shift in spatial variable. More precisely, there is
$\rho _{\mathbf{{n}}}>0$ such that if $(u(x,t),v(x,t))$ is a solution of \eqref{eq_FN_u}-\eqref{eq_FN_v} and
\begin{equation*}
\Vert u(\cdot ,0) - \hat{u}_{\mathbf{{n}}}\Vert _{H^{1}({{\mathbb{R}}})}+
\Vert v(\cdot ,0) - \hat{v}_{\mathbf{{n}}}\Vert _{H^{1}({{\mathbb{R}}})} <
\rho _{\mathbf{{n}}},
\end{equation*}
then
\begin{equation*}
\inf _{y\in {{\mathbb{R}}}}\{\Vert u(\cdot ,t) - \hat{u}_{\mathbf{{n}}}(
\cdot -y)\Vert _{H^{1}({{\mathbb{R}}})}+\Vert v(\cdot ,t) - \hat{v}_{
\mathbf{{n}}}(\cdot -y)\Vert _{H^{1}({{\mathbb{R}}})}\}
\underset{t \to +\infty }{\longrightarrow } 0\,.
\end{equation*}
\end{thm}

As a final remark, there are plenty of unstable standing waves; however
we do not attempt to describe them all. We just state a result in the two-bump
case.

\begin{thm}
\label{thm:mptwobumps}
As in Theorem \ref{thm:multi}, assume (\ref{Hy}) and take a
small $\sigma >0$ and a large constant $\Delta_{\sigma }$. For any positive integer
$n\geq \Delta_{\sigma }$ there exists a solution
$(\check{u}_{n},\check{v}_{n})$ of \eqref{eq_u}-\eqref{eq_v} such that,
for some $X\in {{\mathbb{R}}}$ with
$\vert X-\pi (2n+1)/\omega -\kappa _{+}\vert <\sigma $ and
$\,\kappa _{+}$ as in Theorem \ref{thm:multi}, the following properties
hold:
\smallskip

\noindent
(i)
$\;\Vert (\check{u}_{n},\check{v}_{n}) - (u^{*},v^{*})\Vert _{H^{1}(-
\infty ,\,X/2)} < \sigma  $.%
\smallskip

\noindent
(ii)
$\;\Vert (\check{u}_{n},\check{v}_{n}) - (u_{*},v_{*})(\cdot -X)
\Vert _{H^{1}(X/2,+\infty )} < \sigma  $.%
\smallskip

\noindent
(iii) If $\tau<\gamma^2$ then $(\check{u}_{n},\check{v}_{n})$ is unstable in the following sense:
there is $\epsilon _{n}>0$ such that, for any $\rho >0$, there exist
$\tau _{n}(\rho )>0$ and a solution $(u(x,t),v(x,t))$ of \eqref{eq_FN_u}-\eqref{eq_FN_v} satisfying
\begin{equation*}
\Vert u(\cdot ,0) -\check{u}_{n}\Vert _{H^{1}({{\mathbb{R}}})}+\Vert v(
\cdot ,0) - \check{v}_{n}\Vert _{{H^{1}}({{\mathbb{R}}})} < \rho \,,
\end{equation*}
while for all $\,t\geq \tau _{n}{(\rho )}\,$, one has
\begin{equation*}
\inf _{y\in {{\mathbb{R}}}} \left \{  \Vert u(\cdot ,t) -\check{u}_{n}(
\cdot -y)\Vert _{H^{1}({{\mathbb{R}}})}+\Vert v(\cdot ,t) - \check{v}_{n}(
\cdot -y)\Vert _{{H^{1}}({{\mathbb{R}}})}\right \}  \geq \epsilon _{n}
\,.
\end{equation*}
\end{thm}

The solution $(\check{u}_{n},\check{v}_{n})$ will be found in Section~\ref{unstable} by a mountain-pass type mini-max method. In the proof of
(iii) (as well as in the proof of Theorem \ref{thm:stab}), a crucial ingredient
is Proposition \ref{prop:isolated}, stating that \textit{any} critical point
of $J$ is isolated up to translation.\\

\section{Preliminaries} 
\label{further}
In this section we recall the variational setting \cite{CKM} used to study
$(u^{*},v^{*})$ and discuss related properties, including a reduced functional
$J$ which is bounded from below. In the sequel, we work with affine functional
spaces of the form $H_{w}=w+H^{1}({\mathbb{R}})$, with
$w=0,u_{+},v_{+}, u_{*},u^{*},v_{*}$ or $v^{*}$. For
$a_{u}=0,u_{+},u_{*},u^{*}$ respectively, and
$u \in H_{a_{u}}$, we also denote
${\mathcal{L}} u:= a_{v}+{\mathcal{L}}(u-a_{u}) $, with
$a_{v}=0,v_{+}, v_{*},v^{*}$ respectively. Let us remark that
${\mathcal{L}} u$ is the unique solution, in $H_{a_{v}}$, of the
equation
\begin{equation*}
-v''+\gamma v= u.
\end{equation*}

The operator ${\mathcal{L}}$ has a simple expression:
\begin{lem}
\label{smallness}

Let $G(x)=\frac{1}{2\sqrt{\gamma }}e^{-\sqrt{\gamma }\,\vert x\vert }$. If $u\in H_{a_{u}}$ with $a_{u}\in\{0,u_{+},u^{*},u_{*}\}$, then
$\,\mathcal{L}u=G*u\,$.\medskip

\end{lem}
\begin{proof}
The function $G$ is in $W^{1,1}(\mathbb{R})$ and is a fundamental solution of the operator $-\frac{d^2}{dx^2}+\gamma$. If $u\in H_{a_{u}}$ then it is continuous and bounded on $\mathbb{R}$, so $G*u(x)=\int_{\mathbb{R}}G(x-y)u(y)dy$ is a well-defined function of $x$. This function is of class $C^2$, and solves $-v''+\gamma v=u$. In particular, $G*a_u$ solves $-v''+\gamma v=a_u$. In addition, $\lim_{x\to\pm\infty} G*a_u(x)=\lim_{x\to\pm\infty} a_v(x)$, since $\int_{\mathbb{R}}G=\frac{1}{\gamma}$. So, by the maximum principle, $G*a_u=a_v$. Finally, since $u-a_u\in H^1(\mathbb{R}),$ $G*(u-a_u)\in H^1(\mathbb{R})\,$.
We conclude that $G*u={\mathcal{L}u}\,.$
\eproof\\
\end{proof}

\noindent
Recall from \eqref{La} that the Lagrangian associated with \eqref{eq_u}-\eqref{eq_v} is $L(u,v,u',v')$. Note that the main difference with \cite{BS} is that the present system does not seem to be reducible to
an almost linear, fourth-order system having a simple variational interpretation.
So one has to deal with an indefinite Lagrangian \eqref{La}. Fortunately,
this Lagrangian is concave in $(v,v')$. We exploit this property as follows:
observe that for $u\in H_{a_{u}}$ and $v={\mathcal{L}} u\in H_{a_{v}}$,
\begin{eqnarray*}
\|\frac{d}{dx}(v-a_{v})\|^{2}_{L^{2}({\mathbb{R}})}+\gamma \|v-a_{v}\|^{2}_{L^{2}({
\mathbb{R}})} &=&\int _{-\infty }^{\infty }(u-a_{u%
})(v-a_{v})dx
\\
&\leq & \|u-a_{u}\|_{L^{2}({\mathbb{R}})}\|v-a_{v}\|_{L^{2}({\mathbb{R}})}.
\end{eqnarray*}
Hence there is a $C_{0}>0$ such that
%
\begin{eqnarray}
\label{v}
\| {\mathcal{L}} u-a_{v}\|_{H^{1}({\mathbb{R}})}\leq C_{0}\| u-a_{u}\|_{L^{2}({\mathbb{R}})}.
\end{eqnarray}

Given $\phi \in H^{1}({\mathbb{R}})$, define, for all
$\psi \in H^{1}({\mathbb{R}})$,
\begin{eqnarray*}
I(\psi )&=&\int _{-\infty }^{\infty }(\frac{1}{2}|\psi '|^{2}+
\frac{\gamma }{2}\psi ^{2}-\phi \psi )dx
\end{eqnarray*}

\begin{lem}
\label{coer}
Let $\phi \in H^{1}({\mathbb{R}})$. Then
\begin{eqnarray*}
I(\psi )-I({\mathcal{L}}\phi )&=& \frac{1}{2}\int _{-\infty }^{\infty }\{ (\psi '-({
\mathcal{L}}\phi )')^{2}+\gamma(\psi -{\mathcal{L}}\phi )^{2}\}dx
\end{eqnarray*}
for all $\psi \in H^{1}({\mathbb{R}})$.
\end{lem}
\begin{proof}
It follows from straightforward calculation, by making use of
\begin{eqnarray*}
\int _{-\infty }^{\infty }(({\mathcal{L}}\phi ')^{2}+\gamma ({\mathcal{L}}\phi )^{2}
)dx & = & \int _{-\infty }^{\infty }\phi {\mathcal{L}}\phi dx
\end{eqnarray*}
and
\begin{eqnarray*}
\int _{-\infty }^{\infty }(\psi '{\mathcal{L}}\phi '+\gamma \psi {\mathcal{L}}\phi) dx
&=& \int _{-\infty }^{\infty }\phi \psi dx.
\end{eqnarray*}
\eproof\\
\end{proof}

For $a_{u}\in \{0,u_{+},u^{*},u_{*}\}$ and all $u\in H_{a_{u}}$, define
%
\begin{eqnarray}%
\label{J(u)}
J(u)=\int _{-\infty }^{\infty }L(u,{\mathcal{L}}u,u',{\mathcal{L}}u')dx.
\end{eqnarray}
The next lemma is an immediate consequence of Lemma \ref{coer}.

\begin{lem}
\label{Lu}
Taking $(a_{u},a_{v})$ in $\{(0,0);(u_{+},v_{+});(u^{*},v^{*});(u_{*},v_{*})\}$, if $u\in H_{a_{u}}$ then:%
%
\begin{eqnarray}%
\label{minmax1}
J(u)&=& \max _{v\in H_{a_{v}}}\int _{-\infty }^{\infty }L(u,v,u',v') dx.
\end{eqnarray}
\end{lem}

The following proposition, which was proved in \cite{CKM}, implies that
all the critical values of $J$ on $H_{u^{*}}$ must be positive, since
$d-\frac{1}{\gamma ^{2}}>0$.
%
\begin{prop}
\label{positive}
If $u \in H_{w}$ (with $w=u_{-},\,u_{+},\, u_{*}$ or $u^{*}$) and
$v={\mathcal{L}}u$, then
\begin{equation*}
J(u)=\int ^{\infty }_{-\infty }\left \{  \frac{1}{2}[\left (d-
\frac{1}{\gamma ^{2}}\right )(u')^{2}+\left (v'-\frac{u'}{\gamma }
\right )^{2} +\gamma \left (v-\frac{u}{\gamma }\right )^{2}] +
\frac{1}{4}u^{2}(u-u_{+})^{2} \right \}  dx.
\end{equation*}
\end{prop}
If $J(u)\leq C$ then Proposition \ref{positive} gives an upper bound on
$\Vert u'\Vert _{L^2}$, $\Vert ({\mathcal{L}}u)'\Vert _{L^2}$, $\Vert \gamma{\mathcal{L}}u-u\Vert _{L^2}$ and $\Vert {\rm dist}(u,\{0,u_+\})\Vert _{L^2}$. If, in addition, $u$ is a critical point of
$J$, then, invoking \eqref{eq_u}-\eqref{eq_v}, we obtain an
$L^{\infty }$ bound for $(u,{\mathcal{L}} u,u',{\mathcal{L}} u')$:

%
%
\begin{cor}
\label{bounded}
Assume that $u \in H_{w}$ (with $w=u_{-},\,u_{+},\, u_{*}$ or $u^{*}$) is a critical
point of $J$ and that $J(u)\leq C$ for some $C>0$. Then there is a $C'>0 $, depending only
on $C $, such that
\begin{equation*}
\Vert (u-w,{\mathcal{L}} (u-w),u'-w',{\mathcal{L}} (u'-w'))\Vert _{L^{\infty }}\leq C'.
\end{equation*}
Moreover, in the case  $w\in\{u_{-},\,u_{+}\}$ then $\lim_{C\to 0} C'=0$, while if $w\in\{u_{*},u^{*}\}$, the upper bound $C$ cannot approach zero.
\end{cor}
The choice of parameters made in this paper implies that the Hamiltonian system (HS) given by \eqref{Hm} has the same linearization at
$z_{-}=(u_{-},v_{-},0,0)$ and $z_{+}=(u_{+},v_{+},0,0)$. Moreover these two points are equilibria of saddle-focus type. In other words, the linearized Hamiltonian system at $z_{\pm }$ possesses four complex eigenvalues $\mu e^{i\varpi},\mu e^{-i\varpi},-\mu e^{i\varpi},-\mu e^{-i\varpi}$ with $\mu>0$ and $\varpi\in (0,\frac{\pi}{2})$. Now, we may consider the corresponding linearization of the Lagrangian system \eqref{eq_u}-\eqref{eq_v} at $(u_-,v_-)$ or $(u_+,v_+)$. This linearization takes the form $h''=Bh$ , where $h(x)$ is a column vector with two components and $B$ is a square matrix of order 2 having two simple eigenvalues: $\mu^2e^{2i\varpi}$ and $\mu^2e^{-2i\varpi}$. As a consequence, there exists a change-of-basis matrix
$P$ such that, taking
$Y_{\pm }=P(u-u_{\pm }, v-v_{\pm })^{\mathrm{T}}$, \eqref{eq_u}-\eqref{eq_v}
becomes $Y_{\pm }''=\mu ^{2}R_{2\varpi } Y_{\pm }+O(\vert Y_{\pm }\vert ^{2})$. Here,
$R_{2\varpi }$ is the matrix of a rotation with angle $2\varpi $ in
${\mathbb{R}}^{2}$ and the notation $\mathrm{T}$ means transposition. In the sequel we denote $\lambda$ and $\omega$ the real and imaginary parts of $\mu e^{i\varpi}$.\bigskip

With $z_{\pm }:=(u_{\pm },v_{\pm },0,0)$ being hyperbolic equilibria, a non-constant
solution of \eqref{eq_u}-\eqref{eq_v} cannot stay near them for all values
of $x$. This together with Corollary \ref{bounded} gives a positive
lower bound $\ell $ for the value of $J$ at non-constant critical points. Let
$\vert \cdot \vert $ be the euclidean norm in ${\mathbb{R}}^{2}$. We list below
a number of properties for the Hamiltonian system (HS).\newpage

\begin{prop}
\label{cor:ell}
There exist positive numbers $\rho _{0}$ and $\ell $ such that the
following properties hold.
\begin{itemize}
\item[(i)] Suppose $z(x)=(u(x),v(x),p(x),q(x))$ is a non-constant maximal solution of (HS). Then:\medskip

-If $\vert P(u(x)-u_{\pm },v(x)-v_{\pm })^{\mathrm{T}}\vert <\rho _{0}$ for all
$x\leq 0$, then $z(0)\in W^{u}(z_{\pm })$, and there is an $a_{1} > 0$ such
that
$\vert P(u(a_{1})-u_{\pm },v(a_{1})-v_{\pm })^{\mathrm{T}}\vert =\rho _{0}\,$ and
%
\begin{eqnarray}
\label{increase}
\frac{d}{dx}\vert P(u(x)-u_{\pm },v(x)-v_{\pm })^{\mathrm{T}}\vert \geq
\frac{\lambda }{2}\vert P(u(x)-u_{\pm },v(x)-v_{\pm })^{\mathrm{T}}\vert
\;
\end{eqnarray}
for all $ x \leq a_{1}$.\medskip

-If $\vert P(u(x)-u_{\pm },v(x)-v_{\pm })^{\mathrm{T}}\vert <\rho _{0}$ for all
$x\geq 0$, then $z(0)\in W^{s}(z_{\pm })$, and there is an $a_{2} < 0$ such
that
$\vert P(u(a_{2})-u_{\pm },v(a_{2})-v_{\pm })^{\mathrm{T}}\vert =\rho _{0}\,$ and
%
\begin{eqnarray}
\label{decrease}
\frac{d}{dx}\vert P(u(x)-u_{\pm },v(x)-v_{\pm })^{\mathrm{T}}\vert \leq -
\frac{\lambda }{2}\vert P(u(x)-u_{\pm },v(x)-v_{\pm })^{\mathrm{T}}\vert
\;
\end{eqnarray}
for all $ x \geq a_{2}$.
\item[(ii)] Let $u_{c} $ be a non-constant critical point of $J$ and
$u_{s_{\pm }}=\lim _{x\to \pm \infty } u_{c}(x)$, with
$s_{\pm }\in \{-,+\}$. If $0<\rho\leq\rho _{0}\,,$ there exist
$x_{1}<x_{2}$ such that for any $x\in (-\infty ,x_{1})$,
$P(u_{c},{\mathcal{L}}u_{c})^{\mathrm{T}}(x)$ lies in the open disk
$D_{-}$ of center $P(u_{s_{-}},v_{s_{-}})^{\mathrm{T}}$ with radius
$\rho $ and $P(u_{c},{\mathcal{L}}u_{c})^{\mathrm{T}}(x_{1})$ sits on the
boundary of $D_{-}$, while for any $x\in (x_{2},\infty )$,
$P(u_{c},{\mathcal{L}}u_{c})^{\mathrm{T}}(x)$ lies in the open disk
$D_{+}$ of center $P(u_{s_{+}},v_{s_{+}})^{\mathrm{T}}$ with radius
$\rho $ and $P(u_{c},{\mathcal{L}}u_{c})^{\mathrm{T}}(x_2)$ sits on the
boundary of $D_{+}$.
\item[(iii)] If $u_{c}$ is a non-constant critical point of $J$ in
$H_{w}$ with $w=u_{-},\,u_{+},\, u_{*}$ or $u^{*}$, then
$J(u_{c})\geq \ell $. Moreover
$\Vert u_{c}-u_{\pm }\Vert _{H^{1}}\geq \ell $ in case $ w=u_{\pm }$.
\end{itemize}
\end{prop}

\begin{proof}
We prove (i) when
$\vert P(u(x)-u_{\pm },v(x)-v_{\pm })^{\mathrm{T}}\vert <\rho _{0}$ for all
$x\leq 0$; the other case immediately follows from time reversal. With
$z_{\pm }=(u_{\pm },v_{\pm },0,0)$ being hyperbolic equilibria, if
$\rho _{0}$ is small enough, then, by the Hartman-Grossman Theorem,
$z(0)$ must lie in the local unstable manifold
$W^{u}_{\mathrm{loc}}(z_{\pm })$, which is an embedded submanifold tangent to
the unstable space $E^{u}(z_{\pm })$, by the Stable Manifold Theorem (see
e.g. \cite{Katok-Hasselblatt}).

Let $\dot{z}(x)=(\dot{u}(x),\dot{v}(x), \dot{p}(x),\dot{q}(x))$ be a solution of the linearization of
(HS) at $z_{\pm }$. If $\dot{z}$ lies in $E^{u}(z_{\pm })$, then $\dot{Y}(x):=P(\dot{u}(x),\dot{v}(x))^{\mathrm{T}}$ satisfies the first-order equation $\dot{Y}'=\mu R_\varpi \dot{Y}$ and this implies that
$\frac{d}{dx}\vert \dot{Y}\vert =
\lambda \vert \dot{Y}\vert $. Then, for the nonlinear flow
of (HS), \eqref{increase} follows for small $\rho _{0}$. Note that in \eqref{increase}, $\frac{\lambda }{2}$ could be replaced by any number
$\lambda '<\lambda $, at the expense of choosing a smaller
$\rho _{0}$ when $\lambda '$ is closer to $\lambda $.
\smallskip

Next we prove (ii). For fixed $0<\rho \leq \rho _{0}$, let
$X\in {\mathbb{R}}$ be such that
$\vert P(u_{c}(x)-u_{s_{-}},{\mathcal{L}}u_{c}(x)-v_{s_{-}})^{\mathrm{T}}
\vert <\rho $ for all $x\leq -X$ and
$z(x):=(u_{c}(x-X),{\mathcal{L}}u_{c}(x-X),du_{c}'(x-X),-{\mathcal{L}}u_{c}'(x-X))$.

Since $z$ satisfies the assumptions of (i) corresponding to the case $z(0)\in W^u(z_{s_{-}})$, we may set
\begin{equation*}
x_{1}=\max \{x\leq a_{1}-X\,:\;\forall y<x\,,\;\vert P(u_{c}(y)-u_{s_{-}},{
\mathcal{L}}u_{c}(y)-v_{s_{-}})^{\mathrm{T}}\vert <\rho \, \},
\end{equation*}
where $a_{1}$ was defined by (i). Clearly $x_{1}$ satisfies the required
properties and $x_{2}$ can be found in a similar way.

Finally, since $P(u_{c}(x_{1}),{\mathcal{L}}u_{c}(x_{1}))^{\mathrm{T}}$ has
reached the boundary of $D_{-} $, Corollary \ref{bounded} implies an estimate of the form $J(u_c)\geq \ell$ for some $\ell>0$.
For the same reason, if $u_c\in H_{u_{\pm }}$, $\Vert u_{c}-u_{\pm }\Vert _{H^{1}}$ cannot be too small:
indeed, the $H^1$ norm controls the $L^\infty$ norm. This proves (iii).
\eproof\\
\end{proof}

Next we analyze the behavior of Palais-Smale sequences of $J$. Recall the
classical notion of Palais-Smale sequence.

\begin{defn}
A sequence $\{u_{m}\}$ in $H_{w}$, with
$w\in \{u\pm , u_{*},u^{*}\}$, is called a Palais-Smale sequence for
$J$ if $\{J(u_{m})\}$ is convergent and $J'(u_{m}) \to 0$ in
$H^{-1}(\mathbb{R})$ as $m \to \infty $.
\end{defn}

Due to translation invariance in $x$, a Palais-Smale sequence does not
necessarily have a convergent subsequence for the $H^{1}(\mathbb{R})$ metric.
However, adapting the arguments of \cite{CKM}, we obtain the following
result in the spirit of the concentration-compactness theory by Pierre-Louis
Lions \cite{Li}.

\begin{prop}
\label{Conc-Comp}
Let $\rho _{0}$ be the number introduced in Proposition \ref{cor:ell}, $w\in \{u_{\pm }, u_{*},u^{*}\}$ and
$u_{s_{-}}:=\lim _{x\to -\infty } w (x)$, $s_{-}\in \{-,+\}$. Suppose that
$\{u_{m}\}$ is a Palais-Smale sequence for $J$ in $H_{w}$ such that
$\lim_{m\to \infty } J(u_{m})>0$. Then there is a
$\rho _{1}\in (0,\rho _{0})$ independent of $\{u_{m}\}$ and such that, for any fixed
$\rho \in (0,\rho _{1})$, the following properties hold.
\smallskip
\begin{itemize}
\item[(i)] For $m$ large enough, there exists $x_{m}\in \mathbb{R}$ such that, for all
$x\in (-\infty ,x_{m})$, $P(u_{m},{\mathcal{L}}u_{m})^{\mathrm{T}}(x)$ lies
in the open disk $D_{-}$ of center
$P(u_{s_{-}},{\mathcal{L}}u_{s_{-}})^{\mathrm{T}}$ with radius $\rho $ and
$P(u_{m},{\mathcal{L}}u_{m})^{\mathrm{T}}(x_{m})$ sits on the boundary of
$D_{-}$. Moreover, after extraction,
$\tilde{u}_{m}:=u_{m}(\cdot +x_{m})$ converges in $H^{1}_{\mathrm{loc}}$ to
a non-constant critical point $u^{(1)}$ of $J$, with
$\lim _{x\to -\infty }u^{(1)}(x)=u_{s_{-}}$, while\break
$u_{s_{+}}:=\lim _{x\to +\infty }u^{(1)}(x)$ may differ from
$\lim _{x\to +\infty } w (x)$.\bigskip
\item[(ii)] If $\;\liminf \Vert \tilde{u}_{m}-u^{(1)}\Vert _{H^{1}}>0$, \textit{i.e.}
the convergence of $\{\tilde{u}_{m}\}$ does not hold for the
$H^{1}(\mathbb{R})$ metric, there exist two sequences
$\underbar{t}_{m}\to \underbar{t} \in \mathbb{R}$ and
$\bar{t}_{m}\to +\infty $ such that for all
$x\in (\underbar{t}_{m},\bar{t}_{m})$,
$P(\tilde{u}_{m},{\mathcal{L}}\tilde{u}_{m})^{\mathrm{T}}(x)$ lies in the
open disk $\tilde{D}$ of center
$P(u_{s_{+}},v_{s_{+}})^{\mathrm{T}}$ with radius $\rho $, while
$P(\tilde{u}_{m},{\mathcal{L}}\tilde{u}_{m})^{\mathrm{T}}(\underbar{t}_{m})$
and
$P(\tilde{u}_{m},{\mathcal{L}}\tilde{u}_{m})^{\mathrm{T}}(\bar{t}_{m})$ both
sit on the boundary of $\tilde{D}$. Furthermore along a subsequence,
$\{\tilde{u}_{m}(\cdot +\bar{t}_{m})\}$ converges in
$H^{1}_{\mathrm{loc}}$ to a non-constant critical point $u^{(2)}$ of $J$. Then
$u_{s_{+}}=\lim _{x\to +\infty }u^{(1)}(x)=\lim _{x\to -\infty }u^{(2)}(x)$. Moreover
$$\ell \leq J(u^{(i)})\leq \lim J(u_{m})-\ell \;\,\hbox{ for}\;\, i\in\{1,2\}\,, \ \hbox{ and }\ \limsup \Vert \tilde{u}_{m}-u^{(1)}\Vert _{H^{1}}\geq \ell \,.$$
\end{itemize}
\end{prop}
\begin{proof}
Since $f(\xi ) +\beta \xi = o(\xi )$ as $\xi \to 0$, there exists
$\rho _{1}\in (0,\rho _{0})$ such that
$\vert P(u,v)^{\mathrm{T}}\vert\leq \rho_1$ implies
$-f(u)u \geq \frac{\beta u^{2}}{2} $. Then, by symmetry of the
Lagrangian with respect to $(u_+/2,v_+/2)$,
$\vert P(u-u_{+},v-v_{+})^{\mathrm{T}} \vert \leq \rho _{1} $ implies
$(v_+-f(u))(u-u_{+}) \geq \frac{\beta (u-u_{+})^{2}}{2} $.
\smallskip

To show the existence of $x_{m}$ for $m$ large enough, we assume by contradiction that for any integer $m_0$, one can find $m\geq m_0$ such that
for all $x\in {\mathbb{R}}$,
$P(u_{m},{\mathcal{L}}u_{m})^{\mathrm{T}}(x)$ lies in $D_{-}$. This is possible
only if $w=u_{s_{-}}$. Taking $\rho_1$ small enough, one has ${\rm dist}(P^{-1}D_-,(u_{-s_-},v_{-s_-})^{\mathrm{T}})>0$, so the boundedness of $\{J(u_m)\}$ together with Proposition \ref{positive} implies that
$$\Vert u_{m}-u_{s_-}\Vert _{H^{1}}+\Vert {\mathcal{L}} u_{m}-v_{s_-}\Vert _{H^{1}}=O(1)\;.$$
Now, by definition of $\mathcal{L}\,$, $\hat{e}_{m}:=-({\mathcal{L}}u_{m})''+\gamma {\mathcal{L}}u_{m}-u_{m}=0$. Since $J'(u_m)\to 0$, one also has
$e_{m}:=-d u_{m}''-f(u_{m})+{\mathcal{L}}u_{m}\to 0$ in $H^{-1}$. Combining these estimates, one gets
%
\begin{eqnarray}
\label{em1}
\epsilon _{m}:=<e_{m}, u_{m}-u_{s_{-}}> + <\hat{e}_{m},{\mathcal{L}}u_{m}-v_{s_{-}}>
= o(1).
\end{eqnarray}
On the other hand, by direct calculation,
%
\begin{eqnarray}
\label{em2}
\epsilon _{m}=\int _{-\infty }^{\infty }\{d(u_{m}')^{2}+({\mathcal{L}}u_{m}')^{2}+
\gamma ({\mathcal{L}}u_{m}-v_{s_{-}})^{2}+(v_{s_{-}}-f(u_{m}))(u_{m}-u_{s_{-}}) \}.
\end{eqnarray}
Since we assumed that $P(u_{m},{\mathcal{L}}u_{m})^{\mathrm{T}}(x)$ lies in $D_{-}$ for all
$x\in {\mathbb{R}}$, invoking \eqref{em1}-\eqref{em2} yields
\begin{equation*}
\int _{-\infty }^{\infty }\{d(u_{m}')^{2}+({\mathcal{L}}u_{m}')^{2}+\gamma ({
\mathcal{L}}u_{m}-v_{s_{-}})^{2}+\frac{\beta }{2} (u_{m}-u_{s_{-}})^{2}\}
\leq o(1)\;.
\end{equation*}
This implies that as $m_0\to\infty$, $\Vert u_{m}-u_{s_{-}}\Vert _{H^{1}}\to 0$ and consequently
$J(u_{m})\to 0$, which violates the assumption $\liminf _{m\to \infty } J(u_{m})>0$.
We have thus proved the existence of $x_m$ for $m$ large.%

Next, set $\tilde{u}_{m}:=u_{m}(\cdot +x_{m}) $. Since
$J(\tilde{u}_{m})$ is bounded and $J'(\tilde{u}_{m})\to 0$, using an argument
from \cite{CKM} yields a critical point $u^{(1)}$ of $J$ and a subsequence,
still denoted by $\{\tilde{u}_{m}\}$, such that $\{(\tilde{u}_{m},{\mathcal{L}}\tilde{u}_{m})\}$ converges to $(u^{(1)},{\mathcal{L}} u^{(1)})$ in
$H^{1}_{\mathrm{loc}}(\mathbb{R})$. Note that $u^{(1)}$ satisfies
$\vert P(u^{(1)}(x)-u_{s_{-}},{\mathcal{L}}u^{(1)}(x)-v_{s_{-}})^{
\mathrm{T}}\vert \leq \rho$ for $x\in (-\infty ,0)$ and
$\vert P(u^{(1)}(0)-u_{s_{-}},{\mathcal{L}}u^{(1)}(0)-v_{s_{-}})^{
\mathrm{T}}\vert = \rho$. By Proposition \ref{cor:ell}(i), it is clear
that $(u^{(1)},{\mathcal{L}}u^{(1)})(x)$ converges to
$(u_{s_{-}},v_{s_{-}})$ as $x\to -\infty $.

Assuming that the convergence of $\tilde{u}_{m}$ to $u^{(1)}$ does not
hold for the $H^{1}(\mathbb{R})$ metric, we now prove (ii). Let
$u_{s_{+}}:=\lim _{x\to +\infty }u^{(1)}(x)$ $(s_{+}\in \{-,+\})$. Let
$b_{1}>0$ be such that
%
\begin{eqnarray}
\label{x>b1}
\vert P(u^{(1)}(x)-u_{s_{+}},{\mathcal{L}}u^{(1)}(x)-v_{s_{+}})^{
\mathrm{T}}\vert < \rho /2~ for ~all~ x\geq b_{1}.
\end{eqnarray}
By Proposition \ref{cor:ell}(ii), there is $b_{2} < b_{1}$ such that
$\vert P(u^{(1)}(b_{2})-u_{s_{+}},{\mathcal{L}}u^{(1)}(b_{2})-v_{s_{+}})^{
\mathrm{T}}\vert =\rho _{0}$. Since $\rho _{0}>\rho _{1}>\rho $ and
$\{(\tilde{u}_{m},{\mathcal{L}}\tilde{u}_{m})\}$ uniformly converges to
$(u^{(1)},{\mathcal{L}}u^{(1)})$ on $[b_{2},b_{1}]$, it follows that
$\vert P(\tilde{u}_{m}(b_{2})-u_{s_{+}},{\mathcal{L}}\tilde{u}_{m}(b_{2})-v_{s_{+}})^{
\mathrm{T}}\vert >\rho $ and
$\vert P(\tilde{u}_{m}(b_{1})-u_{s_{+}}, {\mathcal{L}}\tilde{u}_{m}(b_{1})-v_{s_{+}})^{
\mathrm{T}}\vert <\rho $ for $m$ large enough. Set
\begin{equation*}
\underbar{t}_{m}:= \max \{x\leq b_{1}\,:\; \vert P(\tilde{u}_{m}(x)-u_{s_{+}},
{\mathcal{L}}\tilde{u}_{m}(x)-v_{s_{+}})^{\mathrm{T}}\vert \geq \rho \}\,.
\end{equation*}
Then $b_{2} < \underbar{t}_{m} < b_{1}$ and we may assert that
$\underbar{t}_{m}\to \underbar{t}$ after extracting a subsequence. By uniform
convergence of $\{(\tilde{u}_{m},{\mathcal{L}}\tilde{u}_{m})\}$ on compact sets,
we conclude that
$\vert P(u^{(1)}(\underbar{t})-u_{s_{+}},{\mathcal{L}}u^{(1)}(
\underbar{t})-v_{s_{+}})^{\mathrm{T}}\vert =\rho $ and
$\vert P(u^{(1)}(x)-u_{s_{+}},{\mathcal{L}}u^{(1)}(x)-v_{s_{+}})^{
\mathrm{T}}\vert \leq \rho $ for all $x\geq \underbar{t}$.

Set $u^{(2)}_{m}(x):=\tilde{u}_{m}(x)-u^{(1)}(x)+u_{s_{+}}$. Then $\lim_{x\to -\infty}u^{(2)}_{m}(x)=u_{s_{+}}$,
$\{u^{(2)}_{m}\}$ is a Palais-Smale sequence for $J$, and
$\liminf J(u^{(2)}_{m})>0$ follows from the fact that
$\{\tilde{u}_{m}\}$ fails to converge for the $H^{1}(\mathbb{R})$ metric. Let
$\hat{D}$ be the disk of center
$P(u_{s_{+}},v_{s_{+}})^{\mathrm{T}}$ with radius
$(\rho _{1}+\rho )/2$. Replacing $u_m$ by $u^{(2)}_{m}$ and $D_-$ by $\hat{D}$ in the contradiction argument given above, we find that for $m$ sufficiently
large $P(u^{(2)}_{m},{\mathcal{L}}u^{(2)}_{m})^{\mathrm{T}}(x)$ has to exit
from $\hat{D}$. But $(u^{(2)}_{m},{\mathcal{L}}u^{(2)}_{m})$ converges uniformly to
$(u_{s_{+}},{\mathcal{L}}u_{s_{+}})$ on any bounded interval, so if we denote $x^{(2)}_{m}$
the first exit date, we have $\lim_{m\to\infty}x^{(2)}_{m}=+\infty$. As a consequence, 
$P(\tilde{u}_{m},{\mathcal{L}}\tilde{u}_{m})^{\mathrm{T}}(x)$ cannot stay in $\tilde{D}$ for all the values of $x$ in the interval $[b_1,\infty)$, and the first exit date
$$\bar{t}_{m}:=\min \{x\geq b_1\,:\; \vert P(\tilde{u}_{m}(x)-u_{s_{+}},
{\mathcal{L}}\tilde{u}_{m}(x)-v_{s_{+}})^{\mathrm{T}}\vert \geq \rho \}$$
satisfies
$\bar{t}_{m}\to \infty $.

Then,
along a subsequence, $\{u^{(2)}_{m}(\cdot +\bar{t}_{m})\}$ converges in
$H^{1}_{\mathrm{loc}}(\mathbb{R})$ to a non-constant critical point $u^{(2)}$. By
Proposition \ref{cor:ell}(i),
$\lim _{x\to -\infty } (u^{(2)},{\mathcal{L}}u^{(2)})(x) = (u_{s_{+}},v_{s_{+}})$,
since $P(u^{(2)},{\mathcal{L}}u^{(2)})^{\mathrm{T}}(x)$ stays in the closed
disk of center $P(u_{s_{+}},v_{s_{+}})^{\mathrm{T}}$ with radius
$\rho $ for all $x \leq 0$. This completes the proof of (ii).

By the construction of $u^{(2)}_{m}$,
$J(u^{(1)})+\lim J(u^{(2)}_{m})= \lim J(u_{m})$ and\break $J(u^{(2)})\leq \lim J(u^{(2)}_{m})\,,$ so
$J(u^{(1)})+J(u^{(2)})\leq \lim J(u_{m})$. We also have $\Vert u^{(2)}\Vert _{H^{1}}\leq \limsup \Vert \tilde{u}_{m}-u^{(1)}\Vert _{H^{1}}\,.$ Combining these inequalities with Proposition \ref{cor:ell}(iii), we conclude that
$\ell \leq J(u^{(i)})\leq \lim J(u_{m})-\ell $ for $i=1,2$ and $\limsup \Vert \tilde{u}_{m}-u^{(1)}\Vert _{H^{1}}\geq \ell\,$. This ends the proof of Proposition \ref{Conc-Comp}.
\eproof\\
\end{proof}

A consequence of Proposition \ref{Conc-Comp} is the following local Palais-Smale
compactness property:
%
\begin{cor}
\label{local-PS}
Let $\ell $ be the number as defined in Proposition \ref{cor:ell}. Suppose
a Palais-Smale sequence satisfies
$\Vert u_{m}-u_{n}\Vert _{H^{1}}\leq \ell /2$ for all $m,n$, then it has
a convergent subsequence for the $H^{1}(\mathbb{R})$ metric.
\end{cor}

\begin{proof}
The corollary will be proved by contradiction. The idea behind the proof is the following: a Palais-Smale sequence $\{u_m\}$ is always precompact for the $H^1_{\rm loc}$ topology. If it is not precompact for the global $H^1(\mathbb{R})$ metric, then some ``mass" has to escape at infinity. But Proposition \ref{cor:ell} and Proposition \ref{Conc-Comp} imply that this ``mass" has an $H^1$ norm at least equal to $\ell$, so the inequality $\Vert u_{m}-u_{n}\Vert _{H^{1}}\leq \ell /2$ cannot hold for all $m,n$.

More precisely, for the Palais-Smale sequence $\{u_{m}\}$, we pick out the point
$x_{m}$ from each $u_{m}$ as in Proposition \ref{Conc-Comp}(i) and we consider a non-constant critical point $u^{(1)}$ of
$J$ that is the limit (after extraction if necessary) of $\{\tilde{u}_{m}\}:=\{{u}_{m}(\cdot +x_{m})\}$ for the
$H^{1}_{\mathrm{loc}}(\mathbb{R})$ topology. Let us assume that $\{u_{m}\}$ has no convergent subsequence in the $H^1$ metric. Then three cases may occur:

If
$\{x_{m}\}$ has a convergent subsequence of limit $\bar{x}$, then the convergence of $\{\tilde{u}_{m}\}$ to $u^{(1)}$ fails in the $H^{1}(\mathbb{R})$ metric. Hence
$\liminf \Vert \tilde{u}_{m}-u^{(1)}\Vert _{H^{1}}>0$. Applying Proposition \ref{Conc-Comp}(ii), we conclude that
$\limsup \Vert {u}_{m}-u^{(1)}(\cdot-\bar{x})\Vert _{H^{1}}\geq \ell \,.$

Next we turn to the case when, after extraction, $\lim x_{m}=+\infty\,$.
Taking the same definition of the constant $u_{s_-}$ as in Proposition \ref{Conc-Comp}, we see that $\{{u}_{m}\}$ converges to $u_{s_-}$ in $H^{1}_{\mathrm{loc}}(\mathbb{R})$,
and
$\limsup _{m\to \infty } \Vert u_{m}-u_{s_-}\Vert _{H^{1}}\geq \Vert u^{(1)}-u_{s_-}\Vert _{H^{1}}\geq \ell$.

The third case is when, after extraction, $\lim x_{m}=-\infty\,$.
After a new extraction if necessary,
$\{{u}_{m}\}$ converges to a critical point $u$ of $J$, which can be constant or non-constant. Let $u_s:=\lim_{x\to -\infty}u(x)$, with $s\in\{-,+\}$. Then $\limsup _{m\to \infty } \Vert u_{m}-u\Vert _{H^{1}}\geq \Vert u^{(1)}-u_{s}\Vert _{H^{1}}\geq \ell$.

In each of the three cases, we have found, after extraction, that $\{{u}_{m}\}$ converges to a limit $u$ in $H^{1}_{\mathrm{loc}}(\mathbb{R})$, but $\limsup _{m\to \infty } \Vert u_{m}-u\Vert _{H^{1}(\mathbb{R})}\geq \ell$. Now, let $m_0$ be such that $\Vert u_{m_0}-u\Vert _{H^{1}(\mathbb{R})}\geq 3\ell/4$ and let $R_0$ be such that $\Vert u_{m_0}-u\Vert _{H^{1}(-R_0,R_0)}\geq 2\ell/3$.
By local convergence, for $n$ large enough we have $\Vert u_{n}-u\Vert _{H^{1}(-R_0,R_0)}< \ell/6$, hence
$\Vert u_{m_0}-u_n\Vert _{H^{1}(-R_0,R_0)} > \ell/2$. This contradicts the assumption of the corollary.
\eproof\\
\end{proof}

In the construction of multi-front solutions, the trajectories between
two fronts will need to be in good control. Such trajectories are very
close to one of the two stable equilibria with asymptotical behavior being
dominated by the linearized equations. Recall from (\ref{Eg}) that if
$(u,v)$ is a solution of (\ref{eq_u})-(\ref{eq_v}) then
$E(u',v',u,v)$ is constant along the trajectory. In particular, for any
critical point $u$ of $J$, the energy $E(u',{\mathcal{L}}u',u,{\mathcal{L}}u)$ is
identically zero. We now state a lemma in the same spirit as Lemmas 3.1,
A.1, A.2 and A.3 of \cite{BS}.
%
\begin{lem}
\label{cos}
Let $\nu \in (0,\frac{\pi }{2\omega })$ and $\delta $ be small positive numbers.
For $\eta$ and $\zeta$ in ${\mathbb{R}}^{2}$ let us denote $(r_{\eta }, \theta _{\eta })$ and
$(r_{\zeta },\theta _{\zeta })$ the polar coordinates of
$P(\eta _{1}-u_{+},\eta _{2}-v_{+})^{\mathrm{T}}$ and
$P(\zeta _{1}-u_{+},\zeta _{2}-v_{+})^{\mathrm{T}}$, respectively. 
There exists a small radius $\bar{r}>0$ such that, if $r_{\eta }$ and $r_{\zeta }$ are smaller than $\bar{r}$,
then the boundary value problem
\begin{eqnarray}
&& -d u'' = f(u) -v,
\label{3.3a}
\\
&& - v'' = u-\gamma v,
\label{3.3b}
\\
&& (u(0),v(0))=(\eta _{1},\eta _{2}),~~ (u(T),v(T))=(\zeta _{1},
\zeta _{2}),
\label{3.3c}
\end{eqnarray}
has a unique solution in a small neighborhood of $(u_{+},v_{+})\,$. Denoting this solution by
$(\bar{U},\bar{V})_{T,\eta _{1},\eta _{2},\zeta _{1},\zeta _{2}}(\cdot )$,
the following estimate holds on the interval $(0,T)\,$:
$$\vert P((\bar{U},\bar{V})_{T,\eta _{1},\eta _{2},\zeta _{1},\zeta _{2}}(x)-(u_{+},v_{+}))^{\mathrm{T}}\vert<\min\{r_{\eta },r_{\zeta }\}\,.$$

Moreover, if $E_{\eta _{1},\eta _{2},\zeta _{1},\zeta _{2}}(T)$ denotes the
energy of the solution
$(\bar{U},\bar{V})_{T,\eta _{1},\eta _{2},\zeta _{1},\zeta _{2}}$, then the
following properties hold (recalling that the four eigenvalues
of the linearization of (HS) at the saddle-focus equilibrium $(u_{+},v_{+},0,0)\,$ have $\pm \omega$ as imaginary parts):
\begin{itemize}
\item [(i)] For $T$ large enough,
$E_{\eta _{1},\eta _{2},\zeta _{1},\zeta _{2}}(T)>0$ if
$\cos (\theta _{\zeta }-\theta _{\eta }-T\omega +\varphi _{+})>\delta \,,$ while\break
$E_{\eta _{1},\eta _{2},\zeta _{1},\zeta _{2}}(T)<0$ if
$\cos (\theta _{\zeta }-\theta _{\eta }-T\omega +\varphi _{+})<-\delta $. Here,
$\varphi _{+}$ is a phase independent of the parameters.
\item[(ii)] There is a real number $\kappa _{+}$ and, for each
$r\leq \bar{r}/2$, a smaller radius $\epsilon (r)$ proportional to
$r$, such that, if
$\vert P(u^{*}(z)-u_{+},v^{*}(z)-v_{+})^{\mathrm{T}}\vert =r$,
$\vert (\eta _{1}-u^{*}(z),\eta _{2}-v^{*}(z))\vert <\epsilon $,\break
$\vert (\zeta _{1}-u^{*}(z),\zeta _{2}-v^{*}(z))\vert <\epsilon $ and
$n\geq 1/\epsilon $ with $n$ an integer, then
\begin{equation*}
E_{\eta _{1},\eta _{2},\zeta _{1},\zeta _{2}}(\kappa _{+}-2z+2\pi
n/\omega -\nu )>0,
\end{equation*}
\begin{equation*}
E_{\eta _{1},\eta _{2},\zeta _{1},\zeta _{2}}(\kappa _{+}-2z+2\pi
n/\omega +\nu )<0.
\end{equation*}
\end{itemize}
Similar assertions hold with the same $\bar{r}$ and $\epsilon (r)$ when replacing $(u_{+},v_{+})$ by
$(u_{-},v_{-})$, $(u^{*},v^{*})$ by
$(u_{*},v_{*})$ and $\varphi _{+},\kappa _{+}$ by possibly different phases
$\varphi _{-},\kappa _{-}$.

\end{lem}

We refer to \cite{BS} where a very similar statement is proved;
there the existence and local uniqueness
of $(\bar{U},\bar{V})$ follow from Lemma A.3. The sign property (i) of the
energy is a consequence of Lemma A.2, and see Lemma~3.1 for the detail.
Finally, (ii) is a consequence of (i). Note that the formula for the sign
of $E$ in (i) or (ii) is correct for a suitable choice of $P$, for
other choices the sign in front of $T\omega $ has to be changed.

\section{
Isolated 
critical points} \label{iso}
\setcounter{equation}{0}
The aim of this section is to show that all critical points of $J$ are
isolated, up to translations in $x$. As already mentioned in the introduction,
for the construction of multi-front solutions we would just need to know
that the basic heteroclinic $u^{*}$ is isolated up to translations, as
in \cite{BS}. Showing that all critical points are isolated is considerably
more difficult, however this stronger property will be needed for working
out the stability analysis in the last part of this paper.

\begin{prop}
\label{prop:isolated}
Any critical point $u_{c}$ of $J$ is an isolated critical point for the
$H^{1}(\mathbb{R})$ metric, up to translations in $x$.
\end{prop}

Proposition \ref{prop:isolated} will follow from two facts. The first one
is an alternative. Its proof is delicate and relies crucially on the results
of Section~\ref{further} and the real-analyticity of the Hamiltonian:
%
\begin{lem}
\label{lem:alternative}
Either the stable and unstable manifolds of the hyperbolic points
$z_-=(u_{-},v_{-},0,0)$ and $z_+=(u_{+},v_{+},0,0)$ are bounded, or every critical
point of $J$ is isolated up to translation in $x$.
\end{lem}

Note that the four manifolds mentioned in Lemma \ref{lem:alternative} are all bounded (resp. unbounded) if at least one
of them is bounded (resp. unbounded), since the Lagrangian
$L(u,v,u',v')$ is invariant under time reversal and under the symmetry
of center $(u_{+}/2,v_{+}/2)$.
\smallskip

The second fact is stated in the next proposition, in which we find trajectories
which are either on the unstable manifold of $(u_{-},v_{-},0,0)$ or on
the unstable manifold of $(u_{+},v_{+},0,0)$, and that reach a point of
arbitrarily large size. This shows that the stable and unstable manifolds
mentioned in Lemma \ref{lem:alternative} are unbounded, and thus Proposition \ref{prop:isolated} is established:%
%
\begin{prop}
\label{lem:unbounded}
For any $b\in \mathbb{R}$, there is a solution $(u,v)$ of \eqref{eq_u}-\eqref{eq_v}
which satisfies $u(0)=\frac{\gamma }{2}v(0)=b$ and one of the following
conditions:
\begin{itemize}
\item[(i)] $\lim _{x\to -\infty }(u(x),v(x))=(u_{-},v_{-})$;
\item[(ii)] $\lim _{x\to -\infty }(u(x),v(x))=(u_{+},v_{+})$.
\end{itemize}
\end{prop}

\begin{proof}
We use an argument similar to the proof of Theorem 1.1 of \cite{CKM}. For
$s \in \{+,-\}$, let $\hat{v}_{s}$ be a $C^{\infty }$-function on $(-\infty,0]$ such that
%
\begin{eqnarray}
\label{hatv}
\hat{v}_{s}(0)=\frac{2b}{\gamma }, -\hat{v}_{s}''(0)+ \gamma \hat{v}_{s}(0)=b
\text{ and }\hat{v}_{s}(x)=v_{s}\text{ if }x\leq -1.
\end{eqnarray}
For convenience in notation, we define
${\mathbf{H}}_{w} = w + H^{1}_{0}{(-\infty ,0)}$. Set
$\hat{u}_{s}= \gamma \hat{v}_{s}-%
\hat{v}_{s}''$. For $w \in {\mathbf{H}}_{\hat{u}_{s}}%
$, let
%
\begin{eqnarray}
{\hat{J}}(w)=\int ^{0}_{-\infty }\left [\frac{d}{2}({w}^{\prime })^{2}+
\frac{1}{2}w{\hat{\mathcal{L}}}w-\int _{0}^{w} f(\xi )d\xi \right ]dx,
\end{eqnarray}
where ${\hat{\mathcal{L}}}w$ denotes the unique solution of
%
\begin{eqnarray}
&&-v''+\gamma v=w, \text{ }\text{ }\text{ }\text{ } v \in {\mathbf{H}}_{
\hat{v}_{s}}.
\end{eqnarray}
The formula of Proposition \ref{positive} has an easy extension to the
present situation, which can be proved in exactly the same way: if
$w \in {\mathbf{H}}_{\hat{u}_{\varepsilon }}$ then
%
\begin{eqnarray}
{\hat{J}}(w)& = & \int _{-\infty }^{0}\frac{1}{2}\left [ \left (d-
\frac{1}{\gamma ^{2}}\right )(w')^{2} + \left (({\hat{\mathcal{L}}}w)'-
\frac{w'}{\gamma }\right )^{2} +\gamma \left ({\hat{\mathcal{L}}}w-
\frac{w}{\gamma }\right )^{2} \right ] dx
\nonumber
\\
& & ~~ + \int _{-\infty }^{0} \frac{1}{4}(w-u_{-})^{2}(w-u_{+})^{2} dx.
\label{positivebis}
\end{eqnarray}
Set $\hat{c}_{s} = \inf _{w \in {\mathbf{H}}_{\hat{u}_{s}}} {\hat{J}}(w)$. We
distinguish two cases:%
\smallskip

\textbf{Case 1: $\hat{c}_{-}\leq \hat{c}_{+}$.} Pick a sequence
$\{u_{m}\} \subset {\mathbf{H}}_{\hat{u}_{-}}$ such that
$\lim _{m \to \infty }{\hat{J}}(u_{m})=\hat{c}_{-}$. With a priori bounds
on $u_{m}$ derived from \eqref{positivebis} and by passing to a subsequence
if necessary, we may assert, as in the proof of Theorem 1.1 of
\cite{CKM}, that $\{u_{m}\}$ converges to a function $u$ in
$H^{1}_{loc}(-\infty ,0]$. Moreover ${\hat{\mathcal{L}}}u_{m}$ has a limit
$v$ in $H^{1}_{loc}$, and $u(0)=\frac{\gamma }{2}v(0)=b$. Invoking \eqref{positivebis} again, we then prove that
$u\in {\mathbf{H}}_{\hat{u}_{s}}$ for some $s \in \{+,-\}$, and
$v={\hat{\mathcal{L}}} u$, ${\hat{J}}(u)\leq c_{-}$. Now, since
$u\in {\mathbf{H}}_{\hat{u}_{s}}$ and $\hat{c}_{-}\leq \hat{c}_{+}$, we must have
${\hat{J}}(u)\geq c_{s} \geq c_{-}$. Hence ${\hat{J}}(u)= c_{-}$,
$\{u_{m}\}$ strongly converges to $u$ in $H^{1}(-\infty ,0)$ and $u$ is
a minimizer of $\hat{J}$ in ${\mathbf{H}}_{\hat{u}_{-}}$. Thus $(u,v)$ satisfies \eqref{eq_u}-\eqref{eq_v} and (i) holds.%
\smallskip

\textbf{Case 2: $\hat{c}_{-}> \hat{c}_{+}$.} The argument is similar, taking
a sequence $\{u_{m}\} \subset {\mathbf{H}}_{\hat{u}_{+}}$ such that
$\lim _{m \to \infty }{\hat{J}}(u_{m})=\hat{c}_{+}$. As in case 1, after extraction,
$\{u_{m}\}$ strongly converges to a minimizer $u$ of $\hat{J}$ in
${\mathbf{H}}_{\hat{u}_{+}}$ and (ii) holds. Now the proof of Proposition \ref{lem:unbounded} is complete.
\eproof \\
\end{proof}

{\bf Proof of Lemma \ref{lem:alternative} and Proposition \ref{prop:isolated}.} 
We actually just need to prove Lemma \ref{lem:alternative},
since Proposition \ref{prop:isolated} then immediately
follows.%

Let $u_{c}$ be a non-constant critical point of $J$ and
$\lim _{z\to -\infty } u_{c}(x)=u_{s_{-}}$,
$\lim _{z\to +\infty } u_{c}(x)=u_{s_{+}}$,
$s_{-},\,s_{+}\in \{-,+\}$. Note that $z_{s_{-}}=(u_{s_{-}},v_{s_{-}},0,0)$
is a hyperbolic equilibrium of the first order Hamiltonian system (HS) with
Hamiltonian function $H(u,v,p,q)$, which is associated to the second order
Lagrangian system (\ref{eq_u})-(\ref{eq_v}). Here, $p$ denotes the momentum
conjugate to $u$ and $q$ the momentum conjugate to $v$. With the Hamiltonian
$H$ being real analytic, the local unstable manifold
$W_{loc}^{u}(z_{s_{-}})$ of $z_{s_{-}}$ is a real analytic submanifold
of ${\mathbb{R}}^{2}\times {\mathbb{R}}^{2}$. Moreover the unstable space of the linearization
of (HS) at $z_{s_{-}}$ is the graph of a linear map from
${\mathbb{R}}^{2}$ to itself, hence $W_{loc}^{u}(z_{s_{-}})$ is the
graph of a real analytic map $\varphi $ from a small neighborhood
${\mathcal{U}}_{s_{-}}$ of $(u_{s_{-}},v_{s_{-}})$ into ${\mathbb{R}}^{2}$, and
$\varphi (u_{s_{-}},v_{s_{-}})=(0,0)$. Similarly,
${z_{s_{+}}}=(u_{s_{+}},v_{s_{+}},0,0)$ is hyperbolic and its local stable
manifold $W_{loc}^{s}({z_{s_{+}})}$ is the graph of a real analytic map
$\psi =(\psi _{p},\psi _{q})$ from a small neighborhood
${\mathcal{U}}_{s_{+}}$ of $(u_{s_{+}},v_{s_{+}})$ into ${\mathbb{R}}^{2}$, and
$\psi (u_{s_{+}},v_{s_{+}})=(0,0)$.%

Recall that we use the same $\rho _{1}$ as introduced in Proposition \ref{Conc-Comp}.
Also, in Lemma \ref{cos}, we choose $\delta =1/2$ with
the associated $\bar{r} > 0$. To employ Proposition \ref{Conc-Comp}, we
pick $0<\rho \leq \min \{\rho _{1}/2, \bar{r}\}$ such that
$\vert P(u-u_{s_{\pm }},v-v_{s_{\pm }})\vert \leq \rho $ implies
$(u,v)\in {\mathcal{U}}_{s_{\pm }}$. Given $\theta \in {\mathbb{R}}$, we denote
${\mathbf{(u},\mathbf{v)}}^{\mathrm{T}}_{\theta }:=P^{-1}(\rho \cos \theta ,\rho
\sin \theta )^{\mathrm{T}}+(u_{s_{-}},v_{s_{-}})^{\mathrm{T}}$ and
${\mathbf{z}}_{\theta }:=({\mathbf{(u},\mathbf{v)}}_{\theta }, \varphi ({\mathbf{(u},\mathbf{v)}}_{\theta })) $.

Let $v_{c}={\mathcal{L}}u_{c}$. Remembering that $\rho _{1}<\rho _{0}$, we may
take $x_{1},\,x_{2}$ associated to $u_{c}$ and $\rho$ as in Proposition \ref{cor:ell}(ii).
Let $\Phi =(U,V,\Pi _{U},\Pi _{V})$ be the flow of (HS)
at time $T=x_{2}-x_{1}+1$ (both $T$ and $\Phi$ depend on the critical
point $u_{c}$ under consideration). Then we may write\break
$P (u_{c}(x_{1})-u_{s_{-}}, v_{c}(x_{1})-v_{s_{-}})^{\mathrm{T}}=(
\rho \cos \theta _{c},\rho \sin \theta _{c})^{\mathrm{T}}$ for some angle
$\theta _{c}$, and
$P\left (U(\mathbf{z}_{\theta _{c}}),V(\mathbf{z}_{\theta _{c}})\right )^{\mathrm{T}}$
lies in the open disk $D_{+}$ of center
$P(u_{s_{+}},v_{s_{+}})^{\mathrm{T}}$ with radius $\rho $.\smallskip

From now on in this proof, \textit{let us assume that $u_{c}$ is not isolated
up to translations}, for the $H^{1}$ metric in the set of critical points
of $J$. Under this assumption, we are going to show that at least one of the two unstable manifolds $W^u(z_{-})\,,\,W^u(z_{+})$ is bounded. By time reversibility and symmetry with respect to $(u_+/2,v_+/2)$, this will automatically imply that both unstable manifolds are bounded, as well as the two stable manifolds $W^s(z_{\pm})$.\smallskip

Let $u$ be a critical point of $J$ sufficiently close to $u_{c}$
in the $H^{1}(\mathbb{R})$ topology but not equal to a translate of
$u_{c}$. Then the trajectory parametrized by
$(u,{\mathcal{L}}u, d u',-{\mathcal{L}}u')$ in the phase space must contain a point
$\mathbf{z}_{\theta }$ with $\theta $ arbitrarily close, but not equal, to the
angle $\theta _{c}$. Moreover, since
$P\left (U(\mathbf{z}_{\theta _{c}}),V(\mathbf{z}_{\theta _{c}})\right )^{\mathrm{T}}$
lies in the open disk $D_{+}$ of center
$P(u_{s_{+}},v_{s_{+}})^{\mathrm{T}}$ with radius $\rho $, the same
is true for
$P\left (U(\mathbf{z}_{\theta }),V(\mathbf{z}_{\theta })\right )^{\mathrm{T}}$ if $u$ is
close enough to $u_{c}$ in $H^{1}(\mathbb{R})$ topology. So
$\theta $ is a zero of each of the functions
$\chi _{1}(\theta )=\Pi _{U}({\mathbf{z}}_{\theta })- \psi _{p}(U({\mathbf{z}}_{\theta }),V(\mathbf{z}_{\theta }))$ and
$\chi _{2}(\theta )=\Pi _{V}({\mathbf{z}}_{\theta })
- \psi _{q}(U({\mathbf{z}}_{\theta }),V(\mathbf{z}_{\theta }))$.

The above argument shows that $\theta _{c}$ is not an isolated zero of
the real-analytic functions $\chi _{1}$ and $\chi _{2}$, which are defined
in a small open interval ${\mathcal{I}}$ containing $\theta _{c}$. So these
functions are identically zero on ${\mathcal{I}}$, which means that the flow
$\Phi $ sends all the points $\mathbf{z}_{\theta }$ near
$\mathbf{z}_{\theta _{c}}$ to points of $W_{loc}^{s}(z_{s_{+}})$.%

To each $\theta $ in ${\mathcal{I}}$, we associate the solution
$Z_{\theta }$ of the Hamiltonian system with initial value
$Z_{\theta }(0)={\mathbf{z}}_{\theta } $. Let $u_{\theta }(x)$ be the first component
of the vector $Z_{\theta }(x)$ and let $j(\theta ):=J(u_{\theta })$. Then it
is not hard to see that the function
$\theta \in {\mathcal{I}}\to u_{\theta }$ is continuous, and even of class
$C^{1}$, for the $H^{1}({\mathbb{R}})$ metric on the target space. Thus,
by the chain rule, $\frac{d}{d\theta } j(\theta )=0$, and
$j(\theta )=J({\theta _{c}})$. Moreover, all the angles in
${\mathcal{I}}$ correspond to critical points $u_{\theta }$ that are not isolated,
up to translation, in $H^{1}(\mathbb{R})$ topology.

Denote by $\Theta $ the set of angles $\theta $ associated to all non-isolated
critical points $J$ which converge to $u_{s_{-}}$ as $x\to -\infty $. We
have shown that $\Theta $ is an open subset of ${\mathbb{R}}$, and by assumption
it contains $\theta _{c}$. Clearly $\Theta $ is $2\pi $-periodic, and the
above argument shows that $j$ is a $2\pi $-periodic and locally constant
function on $\Theta $. Let ${\mathcal{J}}$ be the maximal open interval in
$\Theta $ containing $\theta _{c}$ and let
$\hat{\theta }:=\sup {\mathcal{J}}$. If $\hat{\theta }=+\infty $, then
${\mathcal{J}}=\Theta ={\mathbb{R}}$ and $j$ is constant on ${\mathbb{R}}$, so
Corollary \ref{bounded} gives a uniform $L^{\infty }$ estimate on \textit{all} the
solutions $Z_{\theta }\,$, so
$W^{u}(z_{s_{-}})=\bigcup _{(\theta ,x)\in [0,2\pi )
\times {\mathbb{R}}} Z_{\theta }(x)$ is bounded. Thus, to complete the proof
of Lemma \ref{lem:alternative}, we just need to study the remaining case
when $\hat{\theta }$ is a finite number.

Then, as $\theta \to \hat{\theta }$, $\theta \in {\mathcal{J}}$,
$Z_{\theta }$ converges in the $C^{1}_{\mathrm{loc}}$ topology to
$Z_{\hat{\theta }}$, by continuous dependence of the solutions of the Hamiltonian
system with respect to initial data. Since
$\hat{\theta }\notin {\mathcal{J}}$, the convergence of $u_{\theta }$ to
$u_{\hat{\theta }}$ does not hold for the $H^{1}({\mathbb{R}})$ metric. As
a consequence, from Proposition \ref{Conc-Comp}(iii),
$u_{\hat{\theta }}$ is a critical point of $j$ satisfying the estimate
$\ell \leq j(u_{\hat{\theta }})\leq j(\theta _{c})-\ell $. From now on,
set
$(u_{s_{1}},v_{s_{1}})=\lim _{x\to +\infty }(u_{\hat{\theta }},{\mathcal{L}}u_{
\hat{\theta }})(x)$, $s_{1}\in \{-,+\}$. Let $\underbar{t} \in \mathbb{R}$ be
such that for all $x>\underbar{t}$,
$P (u_{\hat{\theta }},{\mathcal{L}}u_{\hat{\theta }})^{\mathrm{T}}(x)$ lies
in the open disk $D_{1}$ of center
$P (u_{s_{1}},v_{s_{1}})^{\mathrm{T}}$ with radius $\rho $, while
$P (u_{\hat{\theta }},{\mathcal{L}}u_{\hat{\theta }})^{\mathrm{T}}(
\underbar{t})$ sits on the boundary of $D_{1}$. Then, for $\theta $ close
enough to $\hat{\theta }$, there are two numbers
$\underbar{t}_{\theta }<\bar{t}_{\theta }$ such that for all
$x\in (\underbar{t}_{\theta },\bar{t}_{\theta })$,
$P (u_{\theta },{\mathcal{L}}u_{\theta })^{\mathrm{T}}(x)$ lies in the open
disk $D_{1}$, while
$P (u_{\theta },{\mathcal{L}}u_{\theta })^{\mathrm{T}}(\underbar{t}_{
\theta })$ and
$P (u_{\theta },{\mathcal{L}}u_{\theta })^{\mathrm{T}}(\bar{t}_{\theta })$ both
sit on the boundary of $D_{1}$. Moreover
$ \lim _{\theta \to \hat{\theta }} \underbar{t}_{\theta }=\underbar{t}$ and
$\lim _{\theta \to \hat{\theta }}\bar{t}_{\theta }= +\infty  $. The existence
of $\underbar{t}_{\theta }$ is due to the convergence of $Z_{\theta }$ in
$C^{1}_{\mathrm{loc}}$, and the existence of $\bar{t}_{\theta }$ is due to the
lack of convergence in the $H^{1}({\mathbb{R}})$ metric.

Now, it follows from the Hartman-Grobman theorem that the distance between
$Z_{\theta }(\underbar{t}_{\theta })$ and
$W^{s}_{\mathrm{loc}}(z_{s_{1}})$ tends to zero as
$\theta \to \hat{\theta }$, since
$(\bar{t}_{\theta }-\underbar{t}_{\theta })\to \infty $. As a consequence,
remembering that \eqref{decrease} holds on
$W^{s}_{\mathrm{loc}}(z_{s_{1}})$, we obtain a similar estimate when
$\theta $ is sufficiently close to $\hat{\theta }$:
\begin{equation*}
\frac{d}{dx}\vert P (u_{\theta }(\underbar{t}_{\theta })-u_{s_{1}},{
\mathcal{L}} u_{\theta }(\underbar{t}_{\theta })-v_{s_{1}})^{\mathrm{T}}\vert
\leq -\frac{\lambda }{4}\vert P (u_{\theta }(\underbar{t}_{
\theta })-u_{s_{1}},{\mathcal{L}} u_{\theta }(\underbar{t}_{\theta })-v_{s_{1}})^{
\mathrm{T}}\vert \,.
\end{equation*}
Such an inequality, together with the implicit function theorem gives the
continuity, and even the real-analyticity, of
$\underbar{t}_{\theta }$ as a function of ${\theta }$ on a small interval
$(\hat{\theta }-\varepsilon , \hat{\theta })\subset {\mathcal{J}}$. Similarly,
the distance between $Z_{\theta }(\bar{t}_{\theta })$ and
$W^{u}_{\mathrm{loc}}(z_{s_{1}})$ tends to zero as
$\theta \to \hat{\theta }$, and the continuity of
of $\bar{t}_{\theta }$ can be proved using an inequality
analogous to \eqref{increase}.

So we can define two continuous functions
$\underline{\alpha }\,,\,\bar{\alpha }:\,(\hat{\theta }-\varepsilon ,
\hat{\theta })\to {\mathbb{R}}$ which satisfy
\begin{equation*}
P (u_{\theta }-u_{s_{1}},{\mathcal{L}} u_{\theta }-v_{s_{1}})^{\mathrm{T}}(
\underbar{t}_{\theta })=\rho \left (\cos \underline{\alpha }({\theta }),
\sin \underline{\alpha }({\theta })\right )^{\mathrm{T}},
\end{equation*}
\begin{equation*}
P (u_{\theta }-u_{s_{1}},{\mathcal{L}} u_{\theta }-v_{s_{1}})^{\mathrm{T}}(
\bar{t}_{\theta })=\rho \left (\cos \bar{\alpha }({\theta }),\sin
\bar{\alpha }({\theta })\right )^{\mathrm{T}}.
\end{equation*}
Since
$\lim _{\theta \to \hat{\theta }} (u_{\theta },{\mathcal{L}} u_{\theta })(
\underbar{t}_{\theta })=(u_{\hat{\theta }},{\mathcal{L}} u_{\hat{\theta }})(
\underbar{t})$, $\underline{\alpha }(\theta )$ has a finite limit
$\underline{\alpha }(\hat{\theta })$ as $\theta \to \hat{\theta }$. In order
to study the limit of $\bar{\alpha }$, we are going to use Lemma \ref{cos},
remembering that we have fixed $\delta =1/2$ and chosen
$\rho \leq \bar{r}$. Since $u_{\theta }$ is a critical point of $J$, its energy
is zero, so
$E_{u_{\theta }(\underbar{t}_{\theta }),{\mathcal{L}}u_{\theta }(\underbar{t}_{
\theta }),u_{\theta }(\bar{t}_{\theta }),{\mathcal{L}}u_{\theta }(\bar{t}_{
\theta })}(\bar{t}_{\theta }-\underbar{t}_{\theta })=0 $. Then, from (i)
in Lemma \ref{cos}, we find a phase $\varphi _{1}$ independent of the parameters,
such that, for $\theta $ close enough to $\hat{\theta }$,
\begin{equation*}
\cos (\bar{\alpha }(\theta )-\underline{\alpha }(\theta )-(\bar{t}_{
\theta }-\underbar{t}_{\theta })\omega +\varphi _{1})\in [-1/2,1/2]\,.
\end{equation*}
But $(\bar{t}_{\theta }-\underbar{t}_{\theta })\to \infty $ and
$\underline{\alpha }(\theta )\to \underline{\alpha }(\hat{\theta })$ as
$\theta \to \hat{\theta }$. Moreover $\bar{\alpha }$ depends continuously on
$\theta $. So we must have
$\lim _{\theta \to \hat{\theta }}\bar{\alpha }(\theta ) = \infty $. As a
consequence, given any $\alpha \in [0,2\pi ) $, there is a sequence
$\{\theta _{\alpha }^{(n)}\}$ such that
$\bar{\alpha }(\theta _{\alpha }^{(n)}) = \alpha +2n\pi $ for all $n$ large
enough, and $\theta _{\alpha }^{(n)} \to \hat{\theta }$ as
$n \to \infty $. Passing to a subsequence if necessary, it follows from
Proposition \ref{Conc-Comp} that
$\{u_{\theta _{\alpha }^{(n)}}(\cdot +\bar{t}_{\theta _{\alpha }^{(n)}})
\}$ converges for the local $H^{1}$ topology to a non-constant critical
point of $J$, denoted by $u^{(2)}_{\alpha }$, with
$P (u^{(2)}_{\alpha }-u_{s_{1}},{\mathcal{L}}u^{(2)}_{\alpha }-v_{s_{1}})^{
\mathrm{T}}(0)=\rho \left (\cos \alpha ,\sin \alpha \right )^{\mathrm{T}}$ and
$\ell \leq J(u^{(2)}_{\alpha })\leq J(u_{c})-\ell $. Then Corollary \ref{bounded} gives an $L^{\infty }$ estimate on
$Z^{(2)}_{\alpha }=(u^{(2)}_{\alpha }, {\mathcal{L}} u^{(2)}_{\alpha }, d(u^{(2)}_{\alpha })', -{\mathcal{L}} (u^{(2)}_{\alpha })')$, which is independent of
$\alpha $. Moreover, since the distance between $Z_{\theta _{\alpha }^{(n)} }(\bar{t}_{\theta _{\alpha }^{(n)} })$ and
$W^{u}_{\mathrm{loc}}(z_{s_{1}})$ tends to zero as $n\to\infty$, 
$Z^{(2)}_{\alpha }(0)$ belongs to $W^{u}_{\mathrm{loc}}(z_{s_{1}})$. But $W^{u}_{\mathrm{loc}}(z_{s_{1}})$ is a graph in ${\mathbb{R}}^{2}\times {\mathbb{R}}^{2}$, so we conclude that $W^{u}(z_{s_{1}})=\bigcup _{(\alpha ,x)\in [0,2\pi )\times {\mathbb{R}}}
Z^{(2)}_{\alpha }(x)$. As a consequence, this unstable manifold is bounded. This completes the proof of Lemma \ref{lem:alternative}. So Proposition \ref{prop:isolated} is true.
\eproof \\

Recall that the Lagrangian $L(u,v,u',v')$ is autonomous and $u_{*}$ is
a non-constant solution. By taking a small translation in $x$ if necessary,
from now on we always assume that $u'_{*}(0) \neq 0$. This condition also
holds for $u^{*}$, since it is the reverse orbit of $u_{*}$.\medskip

As a consequence of Proposition \ref{prop:isolated}, we obtain the following
result.
%
\begin{cor}
\label{cor:2-2}
There exist $h_{0},\sigma _{0}>0$ and, for any $0<h<h_{0}$, a radius
$\bar{\sigma }(h)>0$ with $\lim \limits _{h\to 0}\bar{\sigma }(h)=0$, such
that the local sublevel set
\begin{equation*}
{\mathcal{V}}_{h}=\{u\in H_{u^{*}}\;:\; u(0)=u^{*}(0)\,,\; \Vert u-u^{*}
\Vert _{H^{1}({\mathbb{R}})}\leq \sigma _{0}\;and\; J(u)\leq J(u^{*})+h
\}
\end{equation*}
satisfies the following property:
\begin{equation*}
u\in {\mathcal{V}}_{h} \Rightarrow \Vert u-u^{*}\Vert _{H^{1}({\mathbb{R}})}<
\bar{\sigma }(h)\;.
\end{equation*}
\end{cor}

\begin{proof}
From Corollary \ref{local-PS}, for $\sigma _{1}$ small enough, the functional
$J$ satisfies the Palais-Smale condition on the closed ball of center
$u_{*}$ with radius $\sigma _{1}$ of $H^{1}({\mathbb{R}})$-norm. Since
$u'_{*}(0)\neq 0$, by Proposition \ref{prop:isolated} there exists
$\sigma _{0}\leq \sigma _{1}$ such that $u_{*}$ is the unique critical
point of $J$ on the closed ball of center $u_{*}$ with radius
$\sigma _{0}$, so does $u^{*}$. Consequently in this closed ball, Theorem \ref{thm:exist} tells that $u^{*}$ is the unique minimizer of $J$, which
completes the proof.
\eproof\\
\end{proof}

Consider a sufficiently large number $z$ and define
%
\begin{eqnarray}
\label{Vhz}
{\mathcal{V}}_{h,z}:=\{u\in H^{1}(-z,z)\;:\; u\equiv \hat{u}\; \hbox{ on }\; [-z,z]
\; \hbox{ for some }\; \hat{u}\in {\mathcal{V}}_{h}\}\;.
\end{eqnarray}
Now, for $u\in {\mathcal{V}}_{h,z}$ with $h$ small and $z$ large, the functional
$J$ is $C^{2}$ and strictly convex on
\begin{equation*}
{\mathcal{C}}_{u}:=\{\tilde{u}\in H_{u^{*}}\;:\; \tilde{u}\equiv u\;  \hbox{ on }\; [-z,z]
\;  \hbox{ and }\;\Vert \tilde{u}-u^{*}\Vert _{H^{1}({\mathbb{R}})}\leq
\bar{\sigma }(h)\}\;,
\end{equation*}
which is a closed, bounded and convex subset of $H_{u^{*}}$. Indeed, if
$\tilde{u}\in {\mathcal{C}}_{u}$, any other element of $ {\mathcal{C}}_{u}$ near
$\tilde{u}$ is of the form $\tilde{u}+w$ with
$w\equiv 0$ on $[-z,z]$, while $\Vert (\tilde{u}-u_-)(\tilde{u}-u_+)\Vert _{L^{\infty}({\mathbb{R}}\setminus [-z,z])}$ is small. Thus a direct calculation gives, for some $\hat{k}>0 \,$:
\begin{equation*}
D^{2}J(\tilde{u})\cdot w\cdot w=\int _{{\mathbb{R}}} \{ d(w')^{2} - f'(
\tilde{u})w^{2}+({\mathcal{L}}w')^{2}+\gamma ({\mathcal{L}}w)^{2} \} \geq
\hat{k} \Vert w\Vert^2_{H^{1}({\mathbb{R}})}
\end{equation*}
Moreover, if $\tilde{u}\in {\mathcal{C}}_{u}$ satisfies
$\Vert \tilde{u}-u^{*}\Vert _{H^{1}({\mathbb{R}})}= \bar{\sigma }(h)$, then
$J(\tilde{u})>J(u^{*})+h\geq \min _{{\mathcal{C}}_{u}}J $. So the minimizer $b(u)$ of $J$ on $\mathcal{C}_{u}$
does not saturate the constraint $\Vert \tilde{u}-u^{*}\Vert _{H^{1}({\mathbb{R}})}\leq \bar{\sigma }(h)$.
As a consequence, $(b(u),{\mathcal{L}} b(u))$ solves the system (\ref{eq_u})-(\ref{eq_v})
outside the interval $[-z,z]$, and by the implicit function theorem the map
$b:\,{\mathcal{V}}_{h,z}\to H_{u^*} $ constructed in this way is smooth. This
provides a Lyapunov-Schmidt reduction $J_{z}=J\circ b$ of $J$ defined on
${\mathcal{V}}_{h,z}$, and the following corollary holds.

\begin{cor}
\label{cor:_J_z}
For $h_{0}$ small enough, there is $z_{0}>0$ such that if
$h\in (0,h_{0})$ and $z>z_{0}$, then
\begin{equation*}
\rho (h):=\inf \{\Vert J_{z}'(u)\Vert _{(H^{1}(-z,z))^{*}}\,:\; \;u
\in {\mathcal{V}}_{h,z}\; \hbox{ and }\; J_{z}(u)=J(u^{*})+h\} >0\,.
\end{equation*}
\end{cor}
\begin{proof}
Suppose that the assertion of the corollary is false. Then
$\rho (h)=0$ for small $h$ and Corollary \ref{local-PS} implies that a
Palais-Smale sequence converges to a critical point of $J$ in a small ball
of center $u^{*}$ at the critical level $J(u^{*})+h$. Hence there would
exist critical points of $J$ in any small neighborhood of $u^{*}$. This
is contrary to Proposition \ref{prop:isolated}.
\eproof\\
\end{proof}

\section{
 Construction of {multi-front waves}}\label{multibump-section}
 With $u'_{*}(0)\neq 0$, we now get into details about how to construct
the multi-front solutions. Let $h>0$ be small and $\Delta>0$ large (to be determined
later as depending on $h$). Pick an arbitrary integer
$N\geq 1$ and an arbitrary finite sequence of positive integers
${\mathbf{n}}=(n_{i})_{1\leq i\leq N}$ such that $n_{i}\geq \Delta $ for all
$i$. Take $z>0$ large enough so that
$r_-:=\vert P(u^{*}(-z)-u_{-},v^{*}(-z)-v_{-})^{\mathrm{T}}\vert\leq \bar{r}/2$ and
$r_+:=\vert P(u^{*}(z)-u_{+},v^{*}(z)-v_{+})^{\mathrm{T}}\vert\leq \bar{r}/2$, where
$\bar{r}$ is the small radius considered in Lemma \ref{cos}.
\smallskip

Recall ${\mathcal{V}}_{h,z}$ from \eqref{Vhz} and define as follows a smooth map
$b_{\mathbf{n}}$ from $({\mathcal{V}}_{h,z})^{N+1} \times [-\nu ,\nu ]^{N}$ into
$H_w$, with $w=0$ for $N$ odd, $w=u^*$ for $N$ even. To any
$({\mathbf{u}},{\mathbf{x}})=((u_{i})_{0\leq i\leq N},(x_{i})_{1\leq i \leq N})$ in
$({\mathcal{V}}_{h,z})^{N+1} \times [-\nu ,\nu ]^{N}$, we associate a unique
function $u=b_{\mathbf{n}}({\mathbf{u}},{\mathbf{x}})\in H_w$ such that:
\begin{eqnarray*}
&(S_{1})& \forall i \in [0,N]\cap 2{\mathbb{Z}}\,,\; u\equiv u_{i}(
\cdot -C_{i})\;\text{on}\; (C_{i}-z,C_{i}+z),
\\
&(S_{2})& \forall i \in [0,N]\cap (2{\mathbb{Z}}+1)\,,\; u\equiv u_{i}(C_{i}-
\cdot )\;\text{on}\; (C_{i}-z,C_{i}+z),
\\
&(S_{3})& \Vert u - u_{-}\Vert _{H^{1}(-\infty ,-z)}\leq K r_-,
\\
&(S_{4})& \forall i \in [0,N-1]\cap 2{\mathbb{Z}}\,,\; \Vert u - u_{+}
\Vert _{H^{1}(C_{i}+z,C_{i+1}-z)}\leq K r_+,
\\
&(S_{5})& \forall i \in [0,N-1]\cap (2{\mathbb{Z}}+1)\,,\; \Vert u - u_{-}
\Vert _{H^{1}(C_{i}+z,C_{i+1}-z)}\leq K r_-,
\\
&(S_{6})& \Vert u - u_{\pm }\Vert _{H^{1}(C_{N}+z,\infty )}\leq K
r_\pm,\;\text{where}\; \pm =+ \;\text{for $N$ even,} \; \pm =- \;\text{for}
\\
& & \text{$N$ odd,}
\\
&(S_{7})& C_{0}=0\,,\; C_{i}=C_{i-1}+X_{i}\; (1\leq i \leq N)\,,
\\
&(S_{8})& X_{2j} = x_{2j}+\kappa _{-}+\frac{2\pi n_{_{2j}}}{\omega }\,,
\\
&(S_{9})& X_{2j+1} = x_{2j+1}+\kappa _{+} +
\frac{2\pi n_{_{2j+1}}}{\omega }\,,
\\
&(S_{10})& (u,{\mathcal{L}}u) \text{ {satisfies} (\ref{eq_u})-(\ref{eq_v}) on
each of the intervals }(-\infty ,-z]\,,
\\
& & \;[C_{i}+z,C_{i+1}-z]\,,\;[C_{N}+z,+\infty )\;.
\end{eqnarray*}

Taking a large $K$ independent of
$r_\pm$, we claim that for $h$ small and $z$ large enough, conditions $(S_{1})$-$(S_{10})$ determine
$u$ in a unique way, and explain why the corresponding function
$b_{\mathbf{n}}({\mathbf{u}},{\mathbf{x}})%
$ is smooth. Observe that one can define the set
${\mathcal{U}}_{({\mathbf{u}},{\mathbf{x}})}$ consisting of all functions $u$ satisfying
conditions $(S_{1})$-$(S_{6})$. This set is convex, bounded, closed in
the $H^{1}$ topology. Moreover, the controls $(S_{3})$-$(S_{6})$ on
$u$ imply the strict convexity of $J$ restricted to
${\mathcal{U}}_{({\mathbf{u}},{\mathbf{x}})}$. Indeed, if
$u\in {\mathcal{U}}_{({\mathbf{u}},{\mathbf{x}})}$, any other element of
$ {\mathcal{U}}_{({\mathbf{u}},{\mathbf{x}})}$ is of the form $u+w$ with
$w\equiv 0\;$ on $\; A:=\bigcup _{1\leq i \leq N}[C_{i}-z,C_{i}+z]$, while
$\Vert (u-u_-)(u-u_+)\Vert _{L^{\infty}({\mathbb{R}}\setminus A)}$ is small for $K r_\pm$ small. Then
a direct calculation gives
\begin{equation*}
D^{2}J({u})\cdot w\cdot w=\int _{{\mathbb{R}}} \{ d(w')^{2} - f'(u)w^{2}+({
\mathcal{L}}w')^{2}+\gamma ({\mathcal{L}}w)^{2} \} \geq \bar{k} \Vert w\Vert^2_{H^{1}({
\mathbb{R}})}
\end{equation*}
for some $\bar{k}>0 $, exactly as in the proof of Corollary \ref{cor:2-2}.

So $J$ has a unique minimizer in $ {\mathcal{U}}_{({\mathbf{u}},{\mathbf{x}})}$. Moreover
for $K$ large enough, if a function $u$ belongs to
${\mathcal{U}}_{({\mathbf{u}},{\mathbf{x}})}$, and saturates at least one of the constraints
$(S_{3})$-$(S_{6})$ then
$J(u)>\min _{{\mathcal{U}}_{({\mathbf{u}},{\mathbf{x}})}} J$. In conclusion, the minimizer
does not saturate any of the constraints, so it is the only solution of
$(S_{10})$ in ${\mathcal{U}}_{({\mathbf{u}},{\mathbf{x}})}$ and the implicit function
theorem gives a smooth function $b_{\mathbf{n}}$ of $({\mathbf{u}},{\mathbf{x}})$ in the
$H^{1}$ topology.%

Up to this stage, a Lyapunov-Schmidt reduction has been performed, and
the next task is to minimize the reduced functional
${\mathbf{J}}=J\circ b_{\mathbf{n}}$.
Note that the set ${\mathcal{V}}_{h}$ is a bounded, closed sublevel set of the
weakly lower semicontinuous functional $J$, thus ${\mathcal{V}}_{h,z}$ is weakly compact in
$H^{1}(-z,z)$. Moreover one easily checks that ${\mathbf{J}}$ is weakly lower semicontinuous on $({\mathcal{V}}_{h,z})^{N+1} \times [-\nu ,\nu ]^{N}$, which leads to the existence of a minimizer $({\mathbf{\bar{u}}},{\mathbf{\bar{x}}})$ of ${\mathbf{J}}$ on that set.

\begin{lem}
\label{interior}
Given $z$ large, $h$ small and choosing $\Delta$ large enough independently of $N$, if
$n_{i}\geq \Delta$ for every $1\leq i\leq N$, then
${\hat{u}_{\mathbf{n}}}:=b_{\mathbf{n}}({\mathbf{\bar{u}}},{\mathbf{\bar{x}}})$ is a local
minimizer of $J$.
\end{lem}

To prove Lemma \ref{interior}, we introduce the set
${\mathcal{O}}=\bigcup _{({\mathbf{u}},{\mathbf{x}})\in ({\mathcal{V}}_{h,z})^{N+1}
\times [-\nu ,\nu ]^{N}} {\mathcal{U}}_{({\mathbf{u}},{\mathbf{x}})}$ consisting of functions
$u$ satisfying $(S_{1})$-$(S_{6})$ for some
$({\mathbf{u}},{\mathbf{x}})\in ({\mathcal{V}}_{h,z})^{N+1} \times [-\nu ,\nu ]^{N}$.
The next lemma shows that ${\mathcal{O}}$ contains a small ball for the 
$H^{1}({\mathbb{R}})$ metric, with center at $\hat{u}_{\mathbf{{n}}}$. Clearly
$\hat{u}_{\mathbf{{n}}}$ minimizes $J$ on ${\mathcal{O}}$, by virtue of the construction
used in the variational argument, and thus Lemma \ref{interior} is an immediate
consequence.

\begin{lem}
\label{interiorbis}
Choose $z$ large and $h$ small, then $\Delta$ large enough, independently of $N$.
Assume that $n_{i}\geq \Delta$ for all $1\leq i\leq N$. If
$({\mathbf{\bar{u}}},{\mathbf{\bar{x}}})=((\bar{u}_{i})_{0\leq i\leq N},(\bar{x}_{i})_{1
\leq i \leq N})$ is a minimizer of ${\mathbf{J}}$ in
$({\mathcal{V}}_{h,z})^{N+1} \times [-\nu ,\nu ]^{N}$ then
\\
(i) $J_{z}(\bar{u}_{i})<J(u^{*})+h$ for all $0\leq i\leq N$,
\\
(ii) $-\nu <\bar{x}_{i}< \nu $ for all $1\leq i\leq N$.
\end{lem}

The following lemma will be used to prove Lemma \ref{interiorbis}.

\begin{lem}
\label{gradients}
Let $z$ and $h$ be small enough, both being independent of $N$. For any
$\alpha >0$ there exists $\bar{\Delta}(\alpha )$, not depending on $N$, such
that if $n_{i}\geq \bar{\Delta}(\alpha )$ $\forall $ $1\leq i \leq N$ then
\begin{equation*}
\Vert J'_{z}(u_{i})-\partial _{u_{i}}{\mathbf{J}}\Vert _{(H^{1}(-z,z))^{*}}<
\alpha \;,\; \forall ({\mathbf{u}},{\mathbf{x}})\in ({\mathcal{V}}_{h,z})^{N+1}
\times [-\nu ,\nu ]^{N}\;,\;0\leq i\leq N\;.
\end{equation*}
\end{lem}

The proof of Lemma \ref{gradients} is standard (see e.g. \cite{BS}). We
omit it.\\

\noindent
{\bf Proof of Lemma \ref{interiorbis}.}
 We argue indirectly. Suppose that
$J_{z}(\bar{u}_{l})=J(u^{*})+h$ for some $l \in (0,N)$, applying Lemma \ref{gradients} yields
\begin{equation*}
\big <\partial _{{u}_{l}} {\mathbf{J}}({\mathbf{\bar{u}}},{\mathbf{\bar{x}}}),{
\nabla _{H^{1}(-z,z)} J_{z}(\bar{u}_{l})}\big >\,\geq \,
\frac{\rho (h)}{2},
\end{equation*}
with $\rho (h)$ given by Corollary \ref{cor:_J_z}. Then moving
$u_{l}$ slightly in the direction of
$-\nabla _{H^{1}(-z,z)} J_{z}(\bar{u}_{l})$ would decrease
${\mathbf{J}}({\mathbf{u}},{\mathbf{x}})$, which contradicts the minimality of
${\mathbf{J}}({\mathbf{\bar{u}}},{\mathbf{\bar{x}}})$. The proof of (i) is complete.

We next apply Lemma \ref{cos} to prove (ii). Remembering that we have fixed $z$ so that
\begin{equation*}
r_\pm:=\vert P(u^{*}(\pm z),v^{*}(\pm z))^{\mathrm{T}}-P(u_{\pm},v_{\pm})^{\mathrm{T}}\vert \leq \bar{r} /2\;,
\end{equation*}
taking $h$ small enough and $\Delta$ large, we may impose
\begin{equation*}
\vert (\bar{u}_{i}(\pm z),{\mathcal{L}}b_{\mathbf{n}}({\mathbf{\bar{u}}},{\mathbf{\bar{x}}})(C_i\pm (-1)^i z))-(u^{*}(\pm z),v^{*}(
\pm z))\vert <\epsilon(r_\pm) \,
\end{equation*}
with $\epsilon(r_\pm) $ as in Lemma \ref{cos} and the condition $ n_{i}\geq 1/\min\{\epsilon(r_-), \epsilon(r_+)\}$ being
imposed. Suppose $\bar{x}_{l}=-\nu $ for some $l \in (1,N-1)$, it follows
from Lemma \ref{cos} that
\begin{equation*}
\partial _{x_{l}} {\mathbf{J}}({\mathbf{\bar{u}}},{\mathbf{\bar{x}}})=-E_{\eta _{1},
\eta _{2},\zeta _{1},\zeta _{2}}(\kappa _{s}-2z+2\pi n_l/\omega -\nu )<0,
\end{equation*}
where $s=+$ if $l$ is odd, $s=-$ if $l$ is even, $(\eta _{1},
\eta _{2})=(b_{\mathbf{n}}({\mathbf{\bar{u}}},{\mathbf{\bar{x}}}),{\mathcal{L}}b_{\mathbf{n}}({\mathbf{\bar{u}}},{\mathbf{\bar{x}}}))(C_{l-1}+ z)$ and
$(\zeta _{1},
\zeta _{2})=(b_{\mathbf{n}}({\mathbf{\bar{u}}},{\mathbf{\bar{x}}}),{\mathcal{L}}b_{\mathbf{n}}({\mathbf{\bar{u}}},{\mathbf{\bar{x}}}))(C_{l}- z)$. Then increasing
$x_{l}$ slightly would decrease ${\mathbf{J}}$, which again contradicts the
minimality of ${\mathbf{J}}({\mathbf{\bar{u}}},{\mathbf{\bar{x}}})$. Likewise, if
$\bar{x}_{l}=\nu $ we could decrease ${\mathbf{J}}$ by slightly decreasing
$x_{l}$. Now the proof of Lemma \ref{interiorbis} is complete.
\eproof\\

We are now ready to prove the existence result of multi-front solutions,
stated in Theorem \ref{thm:multi}. The stability of such solutions will
be investigated in the next section.\\ 

\noindent
{\bf Proof of Theorem \ref{thm:multi}.}
Impose that $z$ be large enough so that
$K r_\pm < \tilde{\sigma} $, where $\tilde{\sigma}$ is to be chosen later. Then pick
$z$ large and $h$ small enough so that the small number
$\bar{\sigma }(h)$, defined in Corollary \ref{cor:2-2}, is less than
$\tilde{\sigma} $, and so that the conditions of Lemma \ref{interior} are satisfied.

Then the critical point ${\hat{u}_{\mathbf{n}}}$ satisfies the following estimates: 

\begin{eqnarray*}
&\bullet& \Vert \hat{u}_{\mathbf{n}} - u^*\Vert _{H^{1}(-\infty ,X_1-z)}\leq 3\tilde{\sigma},
\\
&\bullet& \forall i \in [1,N-1]\cap 2{\mathbb{Z}}\,,\; \Vert \hat{u}_{\mathbf{n}}(\cdot+C_i) - u^*
\Vert _{H^{1}(-X_i+z,X_{i+1}-z)}\leq 3\tilde{\sigma},
\\
&\bullet& \forall i \in [1,N-1]\cap (2{\mathbb{Z}}+1)\,,\; \Vert \hat{u}_{\mathbf{n}}(\cdot+C_i) - u_*
\Vert _{H^{1}(-X_i+z,X_{i+1}-z)}\leq 3\tilde{\sigma},
\\
&\bullet& \Vert {\hat{u}_{\mathbf{n}}}(\cdot+C_N) - u_{\pm }\Vert _{H^{1}(-X_{N}+z,\infty )}\leq 3\tilde{\sigma},
\;\text{where}\; \pm =+ \;\text{for $N$ even,} \; \pm =- \;\text{for}
\\
& & \text{$N$ odd.}
\end{eqnarray*}

Now we define $\hat{v}_{\mathbf{n}}={\cal L}\hat{u}_{\mathbf{n}}$. Since $\hat{u}_{\mathbf{n}}$ is a critical point of $J$, $(\hat{u}_{\mathbf{n}},\hat{v}_{\mathbf{n}})$ is a solution of \eqref{eq_u}-\eqref{eq_v}. From Lemma \ref{smallness} we also have $\hat{v}_{\mathbf{n}}=G*\hat{u}_{\mathbf{n}}$ with $G(x)=\frac{1}{2\sqrt{\gamma }}e^{-\sqrt{\gamma }\,\vert x\vert }$. From this and the above estimates on $\hat{u}_{\mathbf{n}}$, one can conclude that for $\tilde{\sigma}$ small enough and $\Delta$ large enough (both depending on $\sigma$),
$(\hat{u}_{\mathbf{n}},\hat{v}_{\mathbf{n}})$ satisfies the properties $(a)$ and $(b)$ of Theorem \ref{thm:multi}. This theorem is thus proved.
\eproof\\

\section{Stability}  
\label{ly} 

\setcounter{equation}{0}
From the proof of Theorem \ref{thm:multi}, we know that
$\hat{u}_{\mathbf{{n}}}$ is a local minimizer of $J$ and
$\hat{v}_{\mathbf{{n}}}={\mathcal{L}}\hat{u}_{\mathbf{{n}}}$. To prove Theorem \ref{thm:stab}, we introduce the functional
%
\begin{equation}
\label{LFv}
{\mathcal{E}}(u,v):=J(u)+\frac{\gamma }{2(1+\hat{\delta })}\|v-{\mathcal{L}}u
\|^{2}
\end{equation}
defined for $u \in H_{\hat{u}_{\mathbf{{n}}}}$,
$v \in H_{\hat{v}_{\mathbf{{n}}}}$. Here and in the sequel,
$\Vert \cdot \Vert $ means the $L^{2}({\mathbb{R}})$-norm, and the parameter
$\hat{\delta }>0$ will be chosen later. Proposition \ref{prop:isolated} shows that
$(\hat{u}_{\mathbf{{n}}},\hat{v}_{\mathbf{{n}}})$ is a local minimizer of
${\mathcal{E}}$ for the natural topology of the affine space
$H_{\hat{u}_{\mathbf{{n}}}}\times (\hat{v}_{\mathbf{{n}}}+L^{2}({
\mathbb{R}})\,)$, and it is an isolated critical point of
${\mathcal{E}}$ up to translation in the spatial variable. Also, due to
Corollary \ref{local-PS}, ${\mathcal{E}}$ satisfies the Palais-Smale condition
in a small neighborhood of $(\hat{u}_{\mathbf{{n}}},\hat{v}_{\mathbf{{n}}})$.
\medskip

Consider the Cauchy problem:
%
\begin{eqnarray}
\label{eq_FN_u5}
&&u_{t}-du_{xx}=f(u)-v,
\\
\label{eq_FN_v5}
&&\tau v_{t}- v_{xx}=u-\gamma v,
\end{eqnarray}
with the initial data in the function space $Y=Y_{u} \times Y_{v}$. Here
$Y_{u} = H_{\hat{u}_{\mathbf{{n}}}} \cap C_{b}({\mathbb{R}})$,
$Y_{v} = H_{\hat{v}_{\mathbf{{n}}}} \cap C_{b}({\mathbb{R}})$ and
$C_{b}({\mathbb{R}})$ is the set of bounded uniformly continuous functions
on $\mathbb{R}$. For $(u_{1},v_{1}),(u_{2},v_{2}) \in Y$, define
$||(u_{1},v_{1})-(u_{2},v_{2})||_{Y}=||u_{1}-u_{2}||_{H^{1}({\mathbb{R}})}+||u_{1}-u_{2}||_{L^{\infty }({\mathbb{R}})} + ||v_{1}-v_{2}||_{H^{1}({\mathbb{R}})}+||v_{1}-v_{2}||_{L^{\infty }({\mathbb{R}})}\,.$
Then with minor modification, Theorem 2.1 of
\cite{RS1} (see also \cite{S}, Theorem 14.2) shows that, for given initial
data in $Y$, \eqref{eq_FN_u5}-\eqref{eq_FN_v5} has a unique solution
$(u(\cdot ,t),v(\cdot ,t))$. This solution exists globally in time and
$(u(\cdot ,t),v(\cdot ,t)) \in C([0,\infty ),Y)$. For the proofs, we refer
to \cite{RS1,S} for the detail, including the method of contracting
rectangles. Similar results hold if we work on different function spaces;
for instance, take $Y_{u} = \hat{u}_{\mathbf{{n}}} +C^{2}_{0}({\mathbb{R}})$ and
$Y_{v} = \hat{v}_{\mathbf{{n}}} +C^{{2}}_{0}({\mathbb{R}})$ with the natural topology
inherited from $C^{2}_{0}({\mathbb{R}})$. Here
\begin{equation*}
C^{2}_{0}({\mathbb{R}})=\{w|\frac{d^{j} w}{d^{j}x}\in C_{b}({\mathbb{R}}) ~and
\lim _{|x| \to \infty }\frac{d^{j} w}{d^{j}x}(x)=0~for~0\leq j\leq 2\}.
\end{equation*}

The above results will be used in the proofs of Theorem \ref{thm:stab} and Theorem \ref{thm:mptwobumps}. In the next proposition, we extend a result of \cite{CJM} to show that ${\mathcal{E}}(u,v)$ is a Lyapunov functional for the evolution
flow generated by \eqref{eq_FN_u}-\eqref{eq_FN_v}, for $\hat{\delta }$ chosen small enough.

\begin{prop}
\label{prop:LF}
Assume that $0<\tau <\gamma ^{2}$. Let $\hat{\delta }>0$ such that
$1+\hat{\delta }/2<\gamma ^{2}/{\tau }$. Then for any smooth solution
$(u(x,t), v(x,t))$ of \eqref{eq_FN_u}-\eqref{eq_FN_v},
\begin{eqnarray*}
&&\frac{d}{dt}{\mathcal{E}}(u(\cdot ,t),v(\cdot ,t))\\
&&\leq - \frac{\hat{\delta }}{2(1+\hat{\delta })}\|u_{t}\|^{2}
-\frac{1}{1+\hat{\delta }}\left (\frac{\gamma ^{2}}{\tau }-1-
\frac{\hat{\delta }}{2}\right )\|v-{\mathcal{L}} u \|^{2} 
 -
\frac{\gamma }{(1+\hat{\delta })\tau }\|\frac{\partial }{\partial x}(v-{
\mathcal{L}} u )\|^{2}\\
&&\leq 0\;.
\end{eqnarray*}
\end{prop}
\begin{proof}
Let $w=v-{\mathcal{L}}u $. If $u \in Y_{u}$ and $v \in Y_{v}$ then
$v-{\mathcal{L}} u \in H^{1}({\mathbb{R}})$. By Theorem 2.3 of \cite{RS1} (see
also \cite{S}, Theorem 14.3) and density argument, it suffices to treat
the case $Y_{u} = \hat{u}_{\mathbf{{n}}} +C^{2}_{0}({\mathbb{R}})$ and
$Y_{v} = \hat{v}_{\mathbf{{n}}} +C^{2}_{0}({\mathbb{R}})$. Clearly \eqref{eq_FN_u}-\eqref{eq_FN_v} is equivalent to
%
\begin{eqnarray}
\label{FN1}
&& u_{t}=d u_{xx}+u(u-\beta )(1-u)-{\mathcal{L}}u - w,
\\
\label{FN2}
&& \tau (w_{t}+{\mathcal{L}}u_{t})= w_{xx}-\gamma w.
\end{eqnarray}
In terms of $(u,w)$, we rewrite \eqref{LFv} as
%
\begin{equation}
\label{LF}
{\mathcal{E}}_{1}(u,w):=J(u)+\frac{\gamma }{2(1+\hat{\delta })}\|w\|^{2}.
\end{equation}
Let $(u(x,t),w(x,t))$ be a solution of \eqref{FN1}-\eqref{FN2}. Multiplying \eqref{FN2} by $w$ and integrating by parts, we obtain
%
\begin{eqnarray}
\label{2nd}
{\tau }[(w,w_{t})_{L^{2}}+(w,{\mathcal{L}}u_{t})_{L^{2}}]+\| w_{x} \|^{2}+
\gamma \|w\|^{2}=0.
\end{eqnarray}
Recall that ${\mathcal{L}}$ is a self-adjoint operator in
$L^{2}({\mathbb{R}})\,$. Hence
%
\begin{eqnarray}
\label{sadj}
(w,{\mathcal{L}}u_{t})_{L^{2}}=({\mathcal{L}}w,u_{t})_{L^{2}}.
\end{eqnarray}
By making use of \eqref{2nd}-\eqref{sadj}, a direct calculation gives
\begin{eqnarray*}
&& \frac{d}{dt}{\mathcal{E}}_{1}(u(\cdot ,t),w(\cdot ,t))
\\
&& \quad =-\int _{-\infty }^{\infty } (d u_{xx}+u(u-\beta )(1-u)-{\mathcal{L}}u ) u_{t}dx+
\frac{\gamma }{1+\hat{\delta }}(w,w_{t})_{L^{2}}
\\
&& \quad = -\|u_{t}\|^{2}-(w,u_{t})_{L^{2}} -
\frac{\gamma }{1+\hat{\delta }}\left ( (w,{\mathcal{L}}u_{t})_{L^{2}}+
\frac{1}{\tau }(\| w_{x} \|^{2}+\gamma \|w\|^{2})\right )
\\
&& \quad \leq -\frac{1}{2}\|u_{t}\|^{2}+\frac{1}{2}\|w\|^{2}-
\frac{\gamma }{(1+\hat{\delta })}({\mathcal{L}}w ,u_{t})_{L^{2}}
\\
&& \qquad
-\frac{\gamma ^{2}}{(1+\hat{\delta })\tau }\|w\|^{2}-
\frac{\gamma }{(1+\hat{\delta })\tau }\| w_{x} \|^{2}
\\
&& \quad \leq -\frac{1}{2}\|u_{t}\|^{2}+\frac{1}{2}\|w\|^{2}+
\frac{1}{(1+\hat{\delta })}\|w\|\|u_{t}\| -
\frac{\gamma ^{2}}{(1+\hat{\delta })\tau }\|w\|^{2}-
\frac{\gamma }{(1+\hat{\delta })\tau }\| w_{x} \|^{2}
\\
&& \quad \leq -\frac{1}{2}\left (1-\frac{1}{1+\hat{\delta }}\right )\|u_{t}
\|^{2} +\left (\frac{1}{2}+\frac{1}{2(1+\hat{\delta })}-
\frac{\gamma ^{2}}{(1+\hat{\delta })\tau }\right )\|w\|^{2}
\\
&& \qquad \qquad \qquad -\frac{\gamma }{(1+\hat{\delta })\tau }\| w_{x}
\|^{2}
\\
&& \quad \leq -\frac{\hat{\delta }}{2(1+\hat{\delta })}\|u_{t}\|^{2} -
\frac{1}{1+\hat{\delta }}\left (\frac{\gamma ^{2}}{\tau }-1-
\frac{\hat{\delta }}{2}\right )\|w\|^{2} -
\frac{\gamma }{(1+\hat{\delta })\tau }\| w_{x} \|^{2}\leq 0.
\end{eqnarray*}
Since ${\mathcal{E}}_{1}(u(\cdot ,t),w(\cdot ,t))={\mathcal{E}}(u(\cdot ,t),v(\cdot ,t))\,$, the proposition is proved.
\eproof\\
\end{proof}

\noindent
{\bf Proof of Theorem \ref{thm:stab}.}
By the construction, $\hat{u}_{\mathbf{{n}}}$ is a local minimizer of $J$. From Proposition \ref{prop:isolated}, it is an
isolated critical point of $J$, up to translations.
Moreover, from Corollary \ref{local-PS} and the translation invariance
of $J$, there is $\ell >0$ such that any Palais-Smale sequence
$\{u_{j}\}$ for $J$ satisfying
$\inf _{y\in {\mathbb{R}}}\Vert u_{j}-\hat{u}_{\mathbf{{n}}}(\cdot -y)
\Vert _{H^{1}}\leq \ell /4\,$ is precompact up to translations.\smallskip

Let us take $0<\sigma _{n}\leq \ell /4$ such that
$\hat{u}_{\mathbf{{n}}}$ minimizes $J$ on the set
\begin{equation*}
{\mathcal{X}}_{\mathbf{{n}}}:=\{u\in H_{\hat{u}_{\mathbf{{n}}}}\;:\; \inf _{y\in {\mathbb{R}}}\Vert u -\hat{u}_{\mathbf{{n}}}(\cdot -y)
\Vert _{H^{1}}\leq \sigma _{\mathbf{{n}}}\},
\end{equation*}
and such that any critical point of $J$ in ${\mathcal{X}}_{\mathbf{{n}}}$ is a translate
$\hat{u}_{\mathbf{{n}}}(\cdot -y)$ of $\hat{u}_{\mathbf{{n}}}$. Since
$\sigma _{n}\leq \ell /4$, $J$ satisfies the Palais-Smale condition up
to translation in the set ${\mathcal{X}}_{\mathbf{{n}}}$, and there exists
$h_{\mathbf{{n}}}>0$ such that the local sublevel set
\begin{equation*}
{\mathcal{Y}}_{\mathbf{{n}}}:=\{u\in {\mathcal{X}}_{\mathbf{{n}}}:\; J(u)\leq J(\hat{u}_{\mathbf{{n}}})+h_{
\mathbf{{n}}}\}
\end{equation*}
is a subset of
\begin{equation*}
{\mathcal{X}}'_{\mathbf{{n}}}:=\left \{  u\in H_{\hat{u}_{\mathbf{{n}}}}
\;:\; \inf _{y\in {\mathbb{R}}}\Vert u -\hat{u}_{\mathbf{{n}}}(\cdot -y)\Vert _{H^{1}({\mathbb{R}})}\leq
\frac{\sigma _{\mathbf{{n}}}}{2}\right \}  \,.
\end{equation*}

Next we choose $0<\rho _{\mathbf{{n}}}\leq \sigma _{\mathbf{{n}}}$ such that
${\mathcal{E}}(u,v)\leq J(\hat{u}_{\mathbf{{n}}})+h_{\mathbf{{n}}}$ for all $(u,v)\in H_{\hat{u}_{\mathbf{{n}}}}\times H_{\hat{v}_{\mathbf{{n}}}}$ satisfying
$\;\inf _{y\in {\mathbb{R}}} \{ \Vert u -\hat{u}_{\mathbf{{n}}}(\cdot -y)\Vert _{H^{1}({\mathbb{R}})}+\Vert v-\hat{v}_{
\mathbf{{n}}}(\cdot -y)\Vert _{H^{1}({\mathbb{R}})}\}< \rho _{\mathbf{{n}}}\,.$\smallskip

Let $(u(x,t),v(x,t))$ be a solution of \eqref{eq_FN_u}-\eqref{eq_FN_v}
such that
\begin{equation*}
\Vert u(\cdot ,0) - \hat{u}_{\mathbf{{n}}}\Vert _{H^{1}({{\mathbb{R}}})}+
\Vert v(\cdot ,0) - \hat{v}_{\mathbf{{n}}}\Vert _{H^{1}({{\mathbb{R}}})} <
\rho _{\mathbf{{n}}}\;.
\end{equation*}
Then
$J(u(\cdot ,0))\leq {\mathcal{E}}(u(\cdot ,0),v(\cdot ,0))\leq J(\hat{u}_{
\mathbf{{n}}})+h_{\mathbf{{n}}}$, and
$u(\cdot ,0) \in {\mathcal{X}}_{\mathbf{{n}}}$, so $ u(\cdot ,0)\in {\mathcal{Y}}_{\mathbf{{n}}}$. Since
${\mathcal{E}}$ is a Lyapunov functional, we know that
${\mathcal{E}}(u(\cdot ,t),v(\cdot ,t))\leq J(\hat{u}_{\mathbf{{n}}})+h_{\mathbf{{n}}}$ for all $t\geq 0$.
Hence
%
\begin{eqnarray}
\label{allt}
J(u(\cdot ,t))\leq J(\hat{u}_{\mathbf{{n}}})+h_{\mathbf{{n}}} \text{~for all~} t
\geq 0.
\end{eqnarray}
By continuity of the flow $t\to u(\cdot ,t)$ for the $H^{1}$ topology,
let us prove by contradiction that $u(\cdot ,t)$ stays in ${\mathcal{Y}}_{\mathbf{{n}}}$ for all
$t\geq 0$. Otherwise, there would exist a maximal time $T$ such that
$u(\cdot ,t)\in {\mathcal{Y}}_{\mathbf{{n}}}\;(\forall \,0\leq t\leq T)$. This together
with \eqref{allt} would imply
$\inf _{y\in {\mathbb{R}}}\Vert u(\cdot ,T)-\hat{u}_{\mathbf{{n}}}(\cdot-y)\Vert _{H^{1}({\mathbb{R}})}=
\sigma _{\mathbf{{n}}}\,$, which is not possible since
${\mathcal{Y}}_{\mathbf{{n}}} \subset {\mathcal{X}}'_{\mathbf{{n}}}$.\smallskip

Now, since ${\mathcal{E}}$ is bounded from below, it follows from Proposition \ref{prop:LF} that
%
\begin{eqnarray}
\label{L3}
\int _{0}^{\infty }\|u_{t}(\cdot ,t)\|^{2}_{L^{2}({\mathbb{R}})} dt <
\infty ,
\end{eqnarray}
%
\begin{eqnarray}
\label{L4}
\int _{0}^{\infty }\|v(\cdot ,t)-{\mathcal{L}} u(\cdot ,t)\|^{2}_{L^{2}({
\mathbb{R}})} dt < \infty
\end{eqnarray}
and
%
\begin{eqnarray}
\label{L5}
\int _{0}^{\infty }\|\frac{\partial }{\partial x}(v(\cdot ,t)-{\mathcal{L}} u(
\cdot ,t))\|^{2}_{L^{2}({\mathbb{R}})} dt < \infty .
\end{eqnarray}
Since $u(\cdot ,t)$ stays in ${\mathcal{Y}}_{\mathbf{{n}}}$ for all $t\geq 0$, the
flow $t\to (u(\cdot ,t),v(\cdot ,t))$ is uniformly continuous with respect
to the $H^{1}$ metric, and $t\to u_{t}$ is uniformy continuous with respect
to the $L^{2}$ metric. Consequently
\begin{equation*}
\lim _{t\to \infty }\{\|u_{t}(\cdot ,t)\|_{L^{2}({\mathbb{R}})}+\|v(\cdot ,t)-{
\mathcal{L}} u(\cdot ,t) \|_{H^{1}({\mathbb{R}})}\}=0
\end{equation*}
and $J'(u(\cdot ,t))\to 0$ in $H^{-1}({\mathbb{R}})$. Recall that the
Palais-Smale condition holds in ${\mathcal{Y}}_{\mathbf{{n}}}$ up to translation. Hence the set
$\{u(\cdot ,t),t\geq 0\}$ is precompact up to translation for
the $H^{1}({\mathbb{R}})$ metric, and its limit points are necessarily
critical points of $J$ in ${\mathcal{Y}}_{\mathbf{{n}}}$. Furthermore such a critical
point must be a translate of $\hat{u}_{\mathbf{{n}}}$. As a conclusion,
\begin{equation*}
\lim _{t\to \infty }\inf _{y\in {\mathbb{R}}}\{\Vert u(\cdot ,t)-
\hat{u}_{\mathbf{{n}}}(\cdot -y)\Vert _{H^{1}({\mathbb{R}})}+\|v(\cdot ,t)-
\hat{u}_{\mathbf{{n}}}(\cdot -y)\|_{H^{1}({\mathbb{R}})}\}=0.
\end{equation*}
\eproof\\

\section{
Unstable waves 
} \label{unstable}
\setcounter{equation}{0}
In this section we are going to prove Theorem \ref{thm:mptwobumps}. Some notations will be taken from the proof of Theorem \ref{thm:multi}.  The Lyapunov-Schmidt reduction will be very similar, so the details will be omitted. But note that there will be a shift of $\pi/\omega$ in the condition on the distance between bumps: see (S$'_{5}$) below, compared with (S$_9$). Then a localized mountain-pass principle will be employed instead of minimization, in order to construct unstable two-bump solutions.\smallskip

Starting with the set ${\mathcal{V}}_{h,z}$ defined by \eqref{Vhz}, we construct, for $h$ small and $K,z,n$ large enough, a smooth
map $\check{b}_{n}$ from $({\mathcal{V}}_{h,z})^{2}\times [-\nu ,\nu ]$ to
$H^{1}({\mathbb{R}})$; here to each $(u_{0},u_{1},x)$ in
$({\mathcal{V}}_{h,z})^{2}\times [-\nu ,\nu ]$, we associate the unique function
$u\in H^{1}({\mathbb{R}})$ satisfying the following conditions:
\begin{eqnarray*}
&(S'_{1})& u\equiv u_{0}\;\text{on}\; (-z,z),
\\
&(S'_{2})& u\equiv u_{1}(X-\cdot )\;\text{on}\; (X-z,X+z),
\\
&(S'_{3})& \Vert u - u_{-}\Vert _{H^{1}(-\infty ,-z)}\leq K r_-,
\\
&(S'_{4})& \Vert u - u_{-}\Vert _{H^{1}(X+z,\infty )}\leq K r_+,
\\
&(S'_{5})& X= x+\kappa _{+} +\frac{\pi (2n+1)}{\omega }\,,
\\
&(S'_{6})& (u,{\mathcal{L}}u) \text{ satisfies (\ref{eq_u})-(\ref{eq_v}) on each
of the intervals }(-\infty ,-z]\,,
\\
& & \;[z,X-z]\,,\;[X+z,+\infty )\;.
\end{eqnarray*}

With this definition of $\check{b}_{n}$, we define
$\check{J}:=J\circ \check{b}_{n}$. If
$(u_{0},u_{1},x)$ is a critical point of $\check{J}$ in the interior of
${\mathcal{V}}_{h,z}^{2}\times [-\nu ,\nu ]$, then
$\check{b}_{n}(u_{0},u_{1},x)$ is a critical point of $J\,$.\smallskip

%
%
Moreover, adapting Lemma \ref{cos} to the present situation, we see that,
for each $n$ large enough, there is a small $\mu _{n}>0$ such that if
$x\in [-\nu , -\nu /2] $ then
%
\begin{equation}
\label{outward}
\partial _{x} {\check{J}}(u_{0},u_{1},x)\geq \mu _{n}\,,
\end{equation}
while for $x\in [\nu /2, \nu ] $,
%
\begin{equation}
\label{outwardbis}
\partial _{x} {\check{J}}(u_{0},u_{1},x)\leq -\mu _{n}\,.
\end{equation}

Denoting $u^*_z$ the restriction of $u^*$ to $[-z,z]$, we set
\begin{equation*}
\Gamma :=\{g\in C([-1,1], ({\mathcal{V}}_{h,z})^{2}\times [-\nu ,
\nu ])\;:\; g(\pm 1)=(u^*_z,u^*_z,\pm\nu)\,\}
\end{equation*}
and for each large $n$ assigned in ($S'_5$), we define
\begin{equation*}
c_{n}:= \inf _{g\in \Gamma } \max\{\check{J}\circ g (\tau)\,:\; \tau\in [-1,1]\;\hbox{ and }\;
g(\tau)\in ({\mathcal{V}}_{h,z})^{2}\times [-\nu/2 ,\nu/2 ]\,\}\,.
\end{equation*}

We also introduce the following notations:
\begin{eqnarray*}
&&{\check{J}}^a=J^{-1}((-\infty,a])\,,\;{\check{J}}_a=J^{-1}([a,+\infty))\,,\;\dot{\check{J}}^a=J^{-1}((-\infty,a))\,,\\
&&\mathrm{Cr}({\check{J}},a)=\{(u_0,u_1,x)\in ({\mathcal{V}}_{h,z})^{2}
\times [-\nu ,\nu ] \,:\; {\check{J}}(u_0,u_1,x)= a\;\hbox{ and }\; D{\check{J}}(u_0,u_1,x)=0\}\,.
\end{eqnarray*}

For any $g\in\Gamma\,,$ there is $\tau_0\in [-1,1]$ such that $g(\tau_0)\in ({\mathcal{V}}_{h,z})^{2}\times \{0\}\,.$
As a consequence, $c_n\geq \inf \{\check{J}(u_0,u_1,0):\; (u_0,u_1)\in ({\mathcal{V}}_{h,z})^{2}\,\}\,$.
Moreover, when $n$ tends to infinity, ${\check{J}}(u_0,u_1,x)$ converges uniformly to
$J_z(u_0)+J_z(u_1)$ on $({\mathcal{V}}_{h,z})^{2}\times [-\nu,\nu]\,$, so $c_n\geq 2J(u^*) + o(1)_{n\to\infty}\,.$\smallskip

On the other hand, the function $g_0:\,\tau\in [-1,1]\mapsto (u^*_z,u^*_z,\nu \tau)\,$ belongs to $\Gamma\,$.
As a consequence, $c_n\leq \sup_{\tau\in [-1/2,1/2]}{\check{J}}\circ g_0(\tau)=2J(u^*) + o(1)_{n\to\infty}\,$.
Thus, taking $n$ large enough, we may impose
${\check{J}}^{c_n+h/4} \subset ({\mathcal{V}}_{3h/4,z})^2\times [-\nu,\nu]\,$.\smallskip

Remembering that ${\check{J}}$ is a $C^1$ functional on $F=({\mathcal{V}}_{h,z})^{2}\times [-\nu ,
\nu ]$ satisfying the Palais-Smale condition and that its critical points are isolated, we obtain the following deformation lemma:

\begin{lem}
\label{deformation}

For $\bar{\varepsilon}>0$ small enough, if $V$, $W$ are open neighborhoods of $\mathrm{Cr}({\check{J}},c_n)$ in $F$ such that $\bar{V}\subset W$ and dist$(V,F\setminus W)>0\,$, there exist $\varepsilon\in (0,\bar{\varepsilon})$ and a deformation
$\eta\in C([0,1]\times F,F)$ such that:
\begin{eqnarray*}
&(i)& \eta(0,\cdot)=id_{F}\;,
\\
&(ii)& \eta\left(1,({\check{J}}^{c_n+\varepsilon}\setminus V)\cup (({\mathcal{V}}_{h,z})^{2}\times ([-\nu ,-\nu/2]\cup [\nu/2,
\nu ]))\right)\\
&&\subset {\check{J}}^{c_n-\varepsilon}\cup \left(({\mathcal{V}}_{h,z})^{2}\times ([-\nu ,-\nu/2]\cup [\nu/2,
\nu ])\right)\;,
\\
&(iii)& \eta([0,1]\times \bar{V})\subset W\;,
\\
&(iv)& \eta(t,u_0,u_1,x)=(u_0,u_1,x)\,,\\ 
&&\forall (t,u_0,u_1,x)\in [0,1] \times\left( {\check{J}}_{c_n+\bar{\varepsilon}}
\cup {\check{J}}^{c_n-\bar{\varepsilon}}\cup (({\mathcal{V}}_{h,z})^{2}\times ([-\nu ,-3\nu/4]\cup [3\nu/4,
\nu ]))\right).
\end{eqnarray*}

\end{lem}
\begin{proof} This result is a variant of Lemma 1 in \cite{Hof}, and its proof is similar.
Using \eqref{outward}-\eqref{outwardbis}, one builds a locally Lipschitz pseudo-gradient vector field $\Xi$ for ${\check{J}}$ on ${\check{J}}^{c_n+2\bar{\varepsilon}}\cap {\check{J}}_{c_n-2\bar{\varepsilon}}\setminus \mathrm{Cr}({\check{J}},c_n)$ such that, for any $(u_0,u_1,x)\in {\check{J}}^{c_n+2\bar{\varepsilon}}\cap {\check{J}}_{c_n-2\bar{\varepsilon}}\cap (({\mathcal{V}}_{h,z})^{2}\times ([-\nu ,-\nu/2]\cup [\nu/2,
\nu ]))$, the $x$-component of $\Xi(u_0,u_1,x)$ has the sign of $x$.
Then one builds a locally Lipschitz cut-off function $\chi$ equal to $1$ on ${\check{J}}^{c_n+{\varepsilon}}\cap {\check{J}}_{c_n-{\varepsilon}}\cap (({\mathcal{V}}_{h,z})^{2}\times [-\nu/2 ,\nu/2])\setminus V$ and vanishing on
$V'\cup{\check{J}}_{c_n+\bar{\varepsilon}}\cup {\check{J}}^{c_n-\bar{\varepsilon}}\cup (({\mathcal{V}}_{h,z})^{2}\times ([-\nu ,-3\nu/4]\cup [3\nu/4,
\nu ]))\,$ with $V'$ an open neighborhood of $ \mathrm{Cr}({\check{J}},c_n)$ such that $\bar{V'}\subset V$. The deformation $\eta$ is obtained as the flow of $-\chi\Xi$, after a suitable time reparametrization.
\eproof\\
\end{proof}

Now, using Lemma \ref{deformation}, we can easily adapt Hofer's arguments in \cite{Hof}, and find a ``mountain-pass type" critical point
$(u_{0}^{\sharp },u_{1}^{\sharp },x^{\sharp })$ of ${\check{J}}$ at level $c_n$.\smallskip

Setting $({\check{u}}_{n},{\check{v}}_{n}):=(\check{b}_{n}(u_{0}^{\sharp },u_{1}^{\sharp },x^{\sharp }),{\mathcal{L}}\check{b}_{n}(u_{0}^{\sharp },u_{1}^{\sharp },x^{\sharp }))$, we have thus found a two-bump solution of (\ref{eq_u})-(\ref{eq_v}): for $h$ small and $n$ large enough,
$({\check{u}}_{n},{\check{v}}_{n})$ satisfies Theorem \ref{thm:mptwobumps}$(i)$, $(ii)$.
\smallskip

Moreover, ``mountain-pass type" means that for any neighborhood ${\Omega }$ of
$(u_{0}^{\sharp },u_{1}^{\sharp },x^{\sharp })$, the local sublevel set $\Omega\cup\dot{\check{J}}^{c_n}$ is non-empty and not path-connected. We are now going to  prove the instability of $({\check{u}}_{n},{\check{v}}_{n})\,$ for \eqref{eq_FN_u}-\eqref{eq_FN_v}. \medskip

Let us choose $\epsilon _{n}>0$ such that $J$ satisfies the Palais-Smale
condition, up to translation, in the set
\begin{equation*}
{\mathcal{N}}_{n}:= \{u\in H^{1}({\mathbb{R}}):\; \inf _{y\in {\mathbb{R}}}\Vert u-{\check{u}}_{n}(\cdot-y)
\Vert _{{H^{1}}({{\mathbb{R}}})}\leq 2\epsilon _{n} \}\;,
\end{equation*}
and such that any critical point of $J$ in ${\mathcal{N}}_{n}$ is a translate
of ${\check{u}}_{n}$. Since $(u_{0}^{\sharp },u_{1}^{\sharp },x^{\sharp })$ is
of mountain-pass type, for any $0<\rho <\epsilon _{n}$ there exists
$u_{\rho }\in H^{1}({\mathbb{R}})$ such that
%
\begin{equation}
\label{rho}
\Vert u_{\rho }- {\check{u}}_{n} \Vert _{{H^{1}}({{\mathbb{R}}})}+\Vert {
\mathcal{L}}u_{\rho }- {\check{v}}_{n} \Vert _{{H^{1}}({{\mathbb{R}}})} <
\rho \; \text{ and }\; J(u_{\rho })<c_{n}=J({\check{u}}_{n})\;.
\end{equation}
Let $\{(u(x,t),v(x,t))\,:\,(x,t)\in \mathbb{R}\times \mathbb{R}_+\}$ be the solution of \eqref{eq_FN_u}-\eqref{eq_FN_v}
with initial datum $(u(\cdot ,0),v(\cdot ,0))=(u_{\rho },{\mathcal{L}}u_{\rho })$. Using again the
Lyapunov functional ${\mathcal{E}}$ defined in Section~\ref{ly}, we see that
${\mathcal{E}}(u_{\rho },{\mathcal{L}}u_{\rho })=J(u_{\rho })$ and thus
\begin{equation*}
J(u(\cdot ,t))\leq {\mathcal{E}}(u(\cdot ,t),v(\cdot ,t))\leq J(u_{\rho })<J({\check{u}}_{n})\,,
\; \forall t\geq 0\;.
\end{equation*}

To complete the proof, we argue indirectly to show the existence of
$\tau _{n}(\rho )\geq 0$ such that
\begin{equation*}
\inf _{y\in {\mathbb{R}}} \{\Vert u(\cdot ,t)-{\check{u}}_{n}(\cdot -y)
\Vert _{{H^{1}}({{\mathbb{R}}})}+\Vert v(\cdot ,t)-{\check{v}}_{n}(
\cdot -y)\Vert _{{H^{1}}({{\mathbb{R}}})}\}\geq \epsilon _{n}\,,\;
\forall t\geq \tau _{n}(\rho )\;.
\end{equation*}
Assuming that $\tau _{n}(\rho )$ does not exist, we find a sequence
$\{t_{j}\}$ of times, with $t_{j}\to \infty $ as $j\to \infty $, and a
sequence of translations $\{y_{j}\}$ such that
\begin{equation*}
\Vert u(\cdot ,t_{j})-{\check{u}}_{n}(\cdot -y_{j})\Vert _{{H^{1}}({{
\mathbb{R}}})}+\Vert v(\cdot ,t_{j})-{\check{v}}_{n}(\cdot -y_{j})
\Vert _{{H^{1}}({{\mathbb{R}}})}< \epsilon _{n}\;.
\end{equation*}
By the continuity properties of the flow, we have
\begin{equation*}
\Vert u(\cdot ,t)-{\check{u}}_{n}(\cdot -y_{j})\Vert _{{H^{1}}({{
\mathbb{R}}})}+\Vert v(\cdot ,t)-{\check{v}}_{n}(\cdot -y_{j})\Vert _{{H^{1}}({{
\mathbb{R}}})}\leq 2\epsilon _{n}\,,\;\forall t_{j}\leq t\leq t_{j}+
\theta \;
\end{equation*}
with $\theta >0$ independent of $j$. Since ${\mathcal{E}}$ is bounded from
below, applying Proposition \ref{prop:LF} yields
%
\begin{eqnarray}
\label{L3bis}
\lim _{j\to \infty }\int _{t_{j}}^{t_{j}+\theta } \|u_{t}(\cdot ,t)\|^{2}_{L^{2}({
\mathbb{R}})} dt \,=\, 0\,,
\end{eqnarray}
%
\begin{eqnarray}
\label{L4bis}
\lim _{j\to \infty }\int _{t_{j}}^{t_{j}+\theta } \|v(\cdot ,t)-{\mathcal{L}} u(
\cdot ,t) \|^{2}_{L^{2}({\mathbb{R}})} dt \,=\, 0
\end{eqnarray}
and
%
\begin{eqnarray}
\label{L5bis}
\lim _{j\to \infty }\int _{t_{j}}^{t_{j}+\theta } \|
\frac{\partial }{\partial x}(v(\cdot ,t)-{\mathcal{L}} u(\cdot ,t))\|^{2}_{L^{2}({
\mathbb{R}})} dt \,=\, 0\,.
\end{eqnarray}
Pick $t'_{j}\in [t_{j},\,t_{j}+\theta ]$ such that
\begin{equation*}
\lim _{j\to \infty } \{ \|u_{t}(\cdot ,t'_{j})\|^{2}_{L^{2}({\mathbb{R}})}+
\|v(\cdot ,t'_{j})-{\mathcal{L}} u(\cdot ,t'_{j})\|^{2}_{H^{1}({\mathbb{R}})}
\} \,=\,0\,.
\end{equation*}
Then $u^{(j)}:=u(\cdot ,t'_{j})$ is a Palais-Smale sequence for $J$ in
${\mathcal{N}}_{n}$, and $\lim J(u^{(j)})\leq J(u_\rho)<J({\check{u}}_{n})$. So, using the Palais-Smale condition up to translation in ${\mathcal{N}}_{n}$, we get a critical
point in that set, at a critical level less than $J({\check{u}}_{n})$. This is absurd
since the only critical points of $J$ in ${\mathcal{N}}_{n}$ are the
translates of ${\check{u}}_{n}$. Theorem \ref{thm:mptwobumps} is thus proved.
{\qed}%

\section{Appendix} \label{appendix}

\setcounter{equation}{0}
In this Appendix, we clarify the conditions on the parameters ensuring that
both $(u_{-},v_{-})$ and $(u_{+},v_{+})$ are saddle-focus equilibria at the same energy level.\medskip

First of all, given $\beta\in (0,1)$ and $d,\gamma>0$, the equilibria of (\ref{eq_u})-(\ref{eq_v}) are of the form $(u,\frac{u}{\gamma})$ where $u$ is solution of the polynomial equation $u(u-\beta)(1-u)=\frac{u}{\gamma}\,.$ An obvious solution is $0$, and this gives a first equilibrium $(u_-,v_-)=(0,0)$ having energy zero. If $\gamma>\frac{4}{(1-\beta)^2}$ there are two other equilibria, obtained by solving $u^2-(1+\beta)u+(\beta+\frac{1}{\gamma})=0$. By an elementary computation, one easily checks that one of these additional equilibria has zero energy if and only if $\gamma=\frac{9}{2\beta^2-5\beta+2}\,$, with $\beta<1/2$. Then this equilibrium is $(u_+,v_+)=(\frac{2(\beta+1)}{3},\frac{2(\beta+1)}{3\gamma})$. The other one is $(u_+/2,v_+/2)$ and it is a center of symmetry of our problem. This is better seen if one sets
$k=\frac1{3}(\beta^2-\beta+1)$ and makes the change of variables
%
\begin{equation}
\label{cv}
\left \{
\begin{array}{l}
u=(\beta +1)/3+\sqrt{k}U\,,
\\
v=(\beta +1)/3\gamma +\sqrt{k}V\,.
\end{array}
\right .
\end{equation}
Then, with our choice $\gamma=\frac{9}{2\beta^2-5\beta+2}$, (\ref{eq_u})-(\ref{eq_v}) is equivalent to the system
%
\begin{eqnarray}
\label{eq_U}
&& -dU''=k(U-U^{3})-V,
\\
&& -V''=U-\gamma V.
\label{eq_V}
\end{eqnarray}
The new system is symmetric with respect to the origin. Its nonzero equilibria, which correspond to $(u_{\pm},v_{\pm})$ in the new coordinates, are
$
\pm\left(\sqrt{1-\frac{1}{k\gamma }},\frac{1}{\gamma }\sqrt{1-\frac{1}{k\gamma }}\,\right)$.\smallskip

So, if we impose $0<\beta<1/2$ and $\gamma=\frac{9}{2\beta^2-5\beta+2}\,,$ it is clear that (\ref{eq_u})-(\ref{eq_v}) has the same linearization at $(u_{-},v_{-})$ and $(u_{+},v_{+})$.
Linearizing this system at $(u_{-},v_{-})=(0,0)$, we get an equation of the form $h''=Bh$ where $h(x)$ is a column vector with two components, and
$$B=\begin{bmatrix} 
	d^{-1}\beta & d^{-1}  \\
	-1 & \gamma \\
	
	\end{bmatrix}\;.$$
The equilibria $(u_\pm,v_\pm)$ are of saddle-focus type when $B$ has no real eigenvalue, which means that $(\mathrm{tr} \,B)^2-4\,\mathrm{det} \,B <0\,.$ This condition may be written as follows:
$$ \gamma^2d^2-2(2+\beta\gamma)d +\beta^2 <0\;.  $$
It holds when
$$ \frac{2+\beta\gamma - 2\sqrt{1+\beta\gamma}}{\gamma^2} < d < \frac{2+\beta\gamma + 2\sqrt{1+\beta\gamma}}{\gamma^2} \,.$$
But we also need the functional $J$ to be bounded from below, to ensure the existence of the minimizer $u^*\,$. This imposes the additional condition $d>\frac{1}{\gamma^2}\,,$ as can be seen from Proposition \ref{positive}.\smallskip

In summary, to guarantee that $(u_{-},v_{-})$ and $(u_{+},v_{+})$ are saddle-focus equilibria having the same energy and that $J$ is bounded from below, we require
\begin{equation*}
\begin{split}
&0<\beta <1/2\,,\ \ \gamma =9/(2\beta^2-5\beta+2)\,,\\
&\frac{1}{\gamma^2 }\max(1,2+\beta\gamma-2\sqrt{1+\beta\gamma}) < d < \frac{2+\beta\gamma+2\sqrt{1+\beta\gamma}}{\gamma^2 }\,,
\end{split}
\end{equation*}
as stated in (\ref{Hy}).

\vspace{.3in}
\noindent
{\bf \large Acknowledgments}
The authors thank an anonymous referee who 
gave valuable comments and suggestions to improve the paper.
Research is supported in part by MOST 105-2115-M-007-009-MY3, the Ministry of Science and Technology, Taiwan. Part of the work was done when Chen 
 was visiting the Universit\'e Paris-Dauphine, and S\'er\'e 
  was visiting the National Tsing Hua University and National Center for Theoretical Sciences, Taiwan.

\end{document}